\documentclass[11pt]{article}
\usepackage{times}
\usepackage[top=2.5cm,bottom=2.5cm,left=2.5cm,right=2.5cm]{geometry}
\usepackage[utf8]{inputenc}
\usepackage[T1]{fontenc}
\usepackage{url}
\usepackage{booktabs}
\usepackage{amsfonts}
\usepackage{nicefrac}
\usepackage{microtype}
\usepackage{lipsum}
\usepackage{amsmath}
\usepackage{graphicx}
\graphicspath{{./figure/}}
\usepackage{xcolor}
\usepackage{amsthm}
\newtheorem{theorem}{Theorem}
\newtheorem{lemma}{Lemma}
\newtheorem{remark}{Remark}
\theoremstyle{definition}
\newtheorem{example}{Example}
\usepackage{enumitem}
\numberwithin{equation}{section}
\usepackage{subcaption}
\usepackage{tabu}
\usepackage{multirow}
\usepackage{multicol}
\usepackage{float}
\usepackage{makecell}
\usepackage{tikz}
\usetikzlibrary{shapes.geometric, positioning, calc}
\usepackage[ruled,vlined]{algorithm2e}
\usepackage{amssymb}
\usepackage{setspace}

\tikzset{
	start/.style = {draw, rounded corners, fill=blue!15, text width=6em, align=center},
	proc/.style  = {draw, fill=yellow!15, text width=6em, align=center},
	dec/.style   = {draw, diamond, aspect=2, fill=green!15, text width=5em, align=center},
	end/.style   = {draw, rounded corners, fill=blue!15, text width=6em, align=center},
	arrow/.style = {-Stealth, thick}
}

\usepackage[colorlinks,citecolor=blue]{hyperref}

\title{A Spectral Localization Method for Time-Fractional Integro-Differential Equations with Nonsmooth Data\thanks{This work was partially supported by the National Natural Science Foundation of China under grants 42450192 and 12071343, Natural Science Basic Research Program of Shaanxi (Program No. 2025JC-YBMS-012), and basic research fund of Tianjin University under grant 2025XJ21-0010.}}

\author{
	Lijing Zhao$^{a,b,}$\thanks{E-mail addresses: zhaolj17@nwpu.edu.cn.}
	~~
	Rui Zhao$^{a,b,}$\thanks{E-mail addresses: zhaorui\_rrr@mail.nwpu.edu.cn.}
	~~
	Wenyi Tian$^{c,}$\thanks{Corresponding author, E-mail addresses: twymath@gmail.com.}
	~~
	Yufeng Nie$^{a,b,}$\thanks{E-mail addresses: yfnie@nwpu.edu.cn.}
	\\[10pt]
	\small{$^{a}$ Research Center for Computational Science, Northwestern Polytechnical University, Xi'an 710129, China
	}\\[5pt]
	\small{$^b$ Xi'an Key Laboratory of Scientific Computation and Applied Statistics,}\\
	\small{Northwestern Polytechnical University, Xi'an 710129, China}\\[5pt]
	\small{$^c$ Center for Applied Mathematics and KL-AAGDM,
		Tianjin University, Tianjin 300072, China}\\
}
\date{}

\begin{document}
	\maketitle
	\begin{abstract}
		In this work, we develop a localized numerical scheme with low regularity requirements for solving time-fractional integro-differential equations. First, a fully discrete numerical scheme is constructed. Specifically, for temporal discretization, we employ the contour integral method (CIM) with parameterized hyperbolic contours to approximate the nonlocal operators. For spatial discretization, the standard piecewise linear Galerkin finite element method (FEM) is used. We then provide a rigorous error analysis, demonstrating that the proposed scheme achieves high accuracy even for problems with nonsmooth/vanishing initial values or low-regularity solutions, featuring spectral accuracy in time and second-order convergence in space. Finally, a series of numerical experiments in both 1-D and 2-D validate the theoretical findings and confirm that the algorithm combines the advantages of spectral accuracy, low computational cost, and efficient memory usage.
	\end{abstract}
	
	
	
	
	%
	%
	%
	%
	%
	%
	
	\section{Introduction}
	
	In this paper, we consider the following initial-boundary value problem for time-fractional integro-differential equation with the unknown $u:=u(x,t)$:
	\begin{equation} \label{eq1.1}
		\left\{
		\begin{aligned}
			&{_0^C}\!D_t^\beta u + Au + \int_0^t \kappa_\alpha(t-s)Au(s)\,ds = f, && (x,t) \in \Omega \times (0,T], \\
			&u(x,t) = 0, && (x,t) \in \partial\Omega \times (0,T], \\
			&u(x,0) = u_0(x), && x \in \Omega,
		\end{aligned}\right.
	\end{equation}
	where ${_0^C}\!D_{t}^{\beta}$ with $\beta \in (0,1)$ denotes the Caputo fractional derivative given by
	\begin{equation} \label{eq1.2}
		{_0^C}\!D_{t}^{\beta}u(x,t) := \frac{1}{\Gamma(1-\beta)} \int_{0}^{t} \frac{\partial u(x,\tau)}{\partial \tau} (t-\tau)^{-\beta}\,d\tau,
	\end{equation}
	$\kappa_{\alpha}(t)$ with $\alpha \in (0,1)$ is a memory kernel function defined by
	$$
	\kappa_{\alpha}(t) := {t^{\alpha-1}}/{\Gamma(\alpha)},
	$$
	and $\Gamma(s) = \int_{0}^{\infty} t^{s-1} e^{-t}\,dt$ stands for the Euler Gamma function. In \eqref{eq1.1}, $\Omega$ is a bounded convex polygonal domain in $\mathbb{R}^n(n\geq 1)$ with boundary $\partial\Omega$, $T>0$ is a fixed final time, $A := -\sum_{i=1}^{d} \frac{\partial^2}{\partial x_i^2}$ represents the negative Laplace operator, the function $f=f(x,t)$ is the source term, and $u_0(x)$ is the initial data.
	
	This type of equation, with the nonlocal operator ${_0^C}\!D_{t}^{\beta}$, depicts the anomalous motion characteristics of particles in sub-diffusion process (where the mean squared displacement follows a power law $\langle x^2(t) \rangle \propto t^\beta$ with $\beta \in (0,1)$), and can  better describe the physical phenomena of anomalous diffusion, including the transport of proteins within the cytoskeleton and the diffusion of pollutants in porous media \cite{2008Some}. Additionally, the memory kernel function $\kappa_{\alpha}(t)$ demonstrates a power-law decay characteristics, effectively models the long-time scale dynamics of processes such as the relaxation response in electrochemical systems \cite{Riewe1997Mechanics}. Moreover, from a microscopic perspective, $u(x,t)$ describes the probability density of particles at time $t$ and position $x$, while at the macroscopic scale, it corresponds to the material concentration, and $f(x,t)$ represents the external force that is artificially applied or given. Such integro-differential equation can better describe physical processes that exhibit anomalous diffusion characteristics and long-range temporal dependence, such as telomere motion in cells \cite{mosqueira2020antibody}, diffusion in porous catalysts \cite{caputo2000models}, and transport in nuclear reactor materials \cite{zaeri2017fractional}. Owing to its certain application potential, problem (\ref{eq1.1}) and its nonlinear variants have attracted considerable research attention, with wide-ranging studies spanning both theoretical analysis and numerical computation.
	
	For theoretical analysis, the authors in \cite{Lazarov2015An} studied the existence and regularity of the solutions to problem (\ref{eq1.1}) in Sobolev spaces with Riemann-Liouville fractional derivatives. Regularity estimates for both homogeneous and nonhomogeneous versions of equation (\ref{eq1.1}) were established in \cite{mahata2023nonsmooth} under a range of regularity assumptions on the initial data and source term, with particular emphasis on nonsmooth initial data; corresponding stability estimates were also provided. The authors of \cite{krasnoschok2017semilinear} investigated the semilinear variant of (\ref{eq1.1}) in one dimension, establishing the existence and regularity of solutions within fractional H\"older spaces. These results were later extended to the multidimensional case in \cite{krasnoschok2019semilinear}, which also provided a further analysis of the well-posedness of the solution. Their contributions play a key role in advancing the theoretical understanding of such equations, leading to a relatively well-developed foundation for the analysis.
	
	The numerical computation focus on how to discretize the nonlocal operators, including both the time-fractional derivative operator and the integral operator with a memory kernel. As for the time-fractional derivative term, the popular approximations mainly include the L1 scheme \cite{2006A,lin2007finite}, convolution quadrature method \cite{cuesta2006convolution,deng2019high}, and the improved methods based on these approaches \cite{li2021second,stynes2017error,mustapha2020l1}. For numerically solving problem (\ref{eq1.1}) and its variants:
	\begin{equation} \label{eq1.3}
		D_t^\beta u(x,t) + \gamma Au(x,t) + \int_0^t \kappa_\alpha(t-s)Au(x,s)\,ds = f(x,t),
	\end{equation}
	where $\gamma \geq 0$ and $D_t^\beta$ is the Riemann-Liouville or Caputo fractional derivative, the authors in \cite{mohebbi2017compact} applied a compact finite difference scheme to solve (\ref{eq1.3}) with a particular weakly singular kernel for $\alpha= 1/2$, demonstrating that the method achieves an accuracy of $O(\tau+h^4)$. Based on this, a compact finite difference scheme was proposed in \cite{luo2022numerical} that enhances the temporal convergence order to $O(\tau^2+h^4)$ by using the weighted and shifted Gr\"{u}nwald formula and compact difference operators. Moreover, the authors in \cite{zhou2020alternating} developed an Alternating Direction Implicit (ADI) difference scheme for solving (\ref{eq1.3}), where the Caputo fractional derivative is discretized by the L1 formula, while the integral term is approximated via the convolution quadrature. The proposed scheme is shown to achieve a convergence rate of \( O(\tau^{\min\{2-\beta, 1+\alpha\}} + h^2) \). Then a compact L1-ADI scheme was developed in \cite{wang2021sharp}, which improves the temporal convergence by employing the trapezoidal Proportional-Integral (PI) rule. In addition, the authors of \cite{2009Discontinuous} and \cite{McLean2006ASA} investigated the special case of (\ref{eq1.3}) with $\beta = 1$ and $\gamma = 0$. Under suitable assumptions on the data and for convex or sufficiently smooth spatial domains, they demonstrated that the error converges at the rate of $O(\tau^2+h^2)$.
	
	In the aforementioned studies, the proposed algorithms are inherently nonlocal, leading to a substantial increase in both computational complexity and storage demands as the problem scale expands. Moreover, most of these methods achieve at most second-order convergence in time. Another limitation is their reliance on high solution regularity -- for instance, \( u(t) \in C^2[0,T] \) is required in \cite{zhou2020alternating} -- an assumption often inconsistent with real-world applications. In practice, nonsmooth initial data are frequently encountered, and the actual regularity of solutions tends to be limited. To overcome these challenges, the authors in \cite{mahata2023nonsmooth} presented an analysis of the L1 scheme applied to problem (\ref{eq1.1}), deriving optimal error estimates that remain valid even for nonsmooth initial data such as \( u_0 \in L^2(\Omega)\).
	
	Over the past two decades, a technique known as the Contour Integral Method (CIM), which is based on Laplace transforms and contour integration, has gained prominence \cite{lopez2006spectral,in2011contour,lee2013laplace}. The fundamental idea behind this approach can be traced back to 1979, when Talbot provided a systematic account of its origins and development \cite{talbot1979accurate}. This method expresses the solution of the differential equation via an inverse Laplace transform, which is then approximated by numerical integration along a suitably chosen contour. The idea was first applied to parabolic equations by \cite{sheen2003parallel}, and later extended to fractional-order evolution equations in \cite{mclean2010maximum}. Subsequently, the CIM has been extended to fractional differential equations. For instance, \cite{lee2006parallel} introduced a parallel Laplace transform-based approach for backward parabolic equations; \cite{in2011contour} developd an improved CIM for solving the Black-Scholes and Heston equations; a parabolic contour CIM was applied in \cite{li2022exponential} to handle nonlinear subdiffusion equations with nonsmooth initial data. Collectively, these studies demonstrate that, compared with traditional numerical methods, CIM offers several key advantages: (1) it accommodates solutions with low regularity; (2) the method is localized, meaning the solution at each time level can be computed independently without historical data, which facilitates efficient parallelization; (3) it achieves spectral accuracy in time.
	
	The aim of this paper is to develop a high-performance numerical algorithm -- a localization method -- for the nonlocal problem (\ref{eq1.1}). This method is designed to handle solutions with low regularity while achieving high accuracy, thereby addressing challenges related to nonlocality and nonsmoothness. Specifically, for temporal discretization, we employ the contour integral method (CIM), which represents the solution via an inverse Laplace transform expressed as a contour integral along the left branch of a parameterized hyperbolic contour. This integral is then approximated using the midpoint rule, leading to a semi-discrete CIM scheme in time. For spatial discretization, the piecewise linear finite element method (FEM) is applied. Additionally, we conduct a rigorous error analysis of the fully discrete scheme; the main theoretical results are summarized in Theorems \ref{thm3} and \ref{thm4}.
	
	The main contributions of this work are as follows. We address the nonlocality inherent in both the integral and differential operators, as well as the low regularity of solutions to problem (\ref{eq1.1}):
	\begin{itemize}
		\item An efficient numerical scheme is developed, which uses a parameterized hyperbolic contour CIM for temporal discretization to localize the originally nonlocal problem.
		
		\item The proposed algorithm imposes very low regularity requirements on the solution, enabling robust performance even in the presence of nonsmooth initial data and low regularity solution.
		
		\item Rigorous error analysis demonstrates that the method achieves spectral accuracy in time and second-order convergence in space.
	\end{itemize}
	
	This paper is structured as follows. In Section \ref{section2}, we introduce a fully discrete numerical scheme for problem (\ref{eq1.1}), utilizing the contour integral method (CIM) for semi-discretization in time. Section \ref{section3} is devoted to error analysis, covering both homogeneous problems with nonsmooth initial data and nonhomogeneous problems with vanishing initial data. In Section \ref{section4}, several numerical experiments in 1-D and 2-D are presented, featuring nonsmooth and vanishing initial data, as well as nonsmooth solution, to validate the high performance of the proposed algorithm and to corroborate the theoretical findings. Finally, concluding remarks are provided in Section \ref{section5}.
	
	Throughout this paper \(C\) (with or without a subscript) denotes a positive generic constant that is independent of the solution \(u\), initial data \(u_0\), source function \(f\), time-step size \(\tau\), and mesh size \(h\), but may depend on other parameters $\alpha,\beta$ and \(T\).
	
	\section{Discretization of problem (\ref{eq1.1})}
	\label{section2}
	
	In this section, we introduce the fundamental concepts of the CIM and employ it for semi-discretization in time. Spatial semi-discretization is then carried out using the piecewise linear FEM. Building on these semi-discrete formulations, we present a fully discrete scheme for problem (\ref{eq1.1}) and conclude the section with an accelerated algorithm for its efficient implementation.
	
	\subsection{Contour Integral Method (CIM)}
	
	Let \( u(t) \) be a real-valued function. Under certain conditions-specifically, if there exist constants \( \sigma_0 > 0 \) and \( M > 0 \) such that \( |u(t)| \leq M e^{\sigma_0 t} \), then its Laplace transform \( \widehat{u}(z) := \mathcal{L}\{u(t)\}(z) \) exists for \( \operatorname{Re}(z) > \sigma_0 \). In that case, \( u(t) \) can be expressed via the inverse Laplace transform as
	\begin{equation} \label{eq2.1}
		u(t) = \frac{1}{2\pi i} \int_{\sigma_0-i\infty}^{\sigma_0+i\infty} e^{zt} \widehat{u}(z)\,dz, \quad \operatorname{Re}(z)>\sigma_0.
	\end{equation}
	To compute \( u(t) \), it then suffices to choose a suitable numerical method for approximating this contour integral. In order to avoid the high-frequency oscillations of the exponential factor along the vertical line \( z = \sigma_0 + iy \), where \( e^{zt} = e^{\sigma_0 t}\cdot(\cos(yt)+i\sin(yt)) \) for \( -\infty < y < +\infty \), we employ the CIM, which offers both efficiency and simplicity.
	
	The key idea of the CIM is to perform an equivalent transformation on the integral path: according to Cauchy's Integral Theorem, the original path of integration for the inverse Laplace transform, which is a vertical line extending from $-\infty$ to $+\infty$, can be transformed into a new contour $\Gamma$, so that the integral (\ref{eq2.1}) can be rewritten as
	\begin{equation} \label{eq2.3}
		u(t) = \frac{1}{2\pi i} \int_{\Gamma} e^{zt} \widehat{u}(z)\,dz.
	\end{equation}
	Here, $\Gamma$ can be any contour in the complex plane that starts from $-\infty$ in the third quadrant, encircles the origin, 
	and returns to $-\infty$ in the second quadrant, as shown in Figure \ref{contour}, such that all the singularities of $\widehat{u}(z)$ lie in the left half plane of it. At this point, the exponential factor $e^{zt}$ leads to rapid decay of the integrand along the contour $\Gamma$, which will greatly enhance the efficiency and accuracy of the quadrature.
	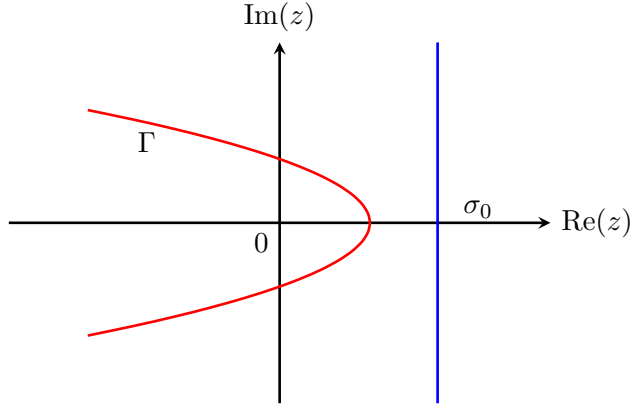
\begin{figure}[h!]
		\centering
		\begin{tikzpicture}[>=stealth,scale=0.6,line width=1.0pt]
			
			\draw[->] (-6,0) -- (6,0) node[right] {$\operatorname{Re}(z)$};
			\draw[->] (0,-4) -- (0,4) node[above] {$\operatorname{Im}(z)$};
			
			\node[below left] at (0,0) {$0$};
			
			\draw[red,domain=-2.5:2.5,smooth,variable=\y]
			plot ({2-\y*\y},{\y});
			\node[left] at (-2.5,1.8) {$\Gamma$};
			
			\draw[blue] (3.5,-4) -- (3.5,4);
			\node[right=2mm] at (3.5,0.3) {$\sigma_0$};
			
		\end{tikzpicture}
		\caption{Illustration of the integration path transformation.}
		\label{contour}
	\end{figure}
	
	In most cases, $\Gamma$ is typically chosen as Talbot's contour \cite{weideman2006optimizing}, parabolic contours \cite{weideman2019gauss,weideman2007parabolic}, or hyperbolic contours \cite{weideman2007parabolic,lopez2006spectral,sheen2003parallel}. In the following text, we denote it as
	$$
	\Gamma: z = z(\phi), \quad -\infty < \phi < \infty,
	$$
	indicating that the integration path is a specifically defined curve that can be transformed into a horizontal line, yielding
	\begin{equation} \label{eq2.4}
		u(t) = I := \int_{-\infty}^{\infty} g(\phi,t)\,d\phi,
	\end{equation}
	where
	\begin{equation} \label{eq.2.5}
		g(\phi,t)=\frac{1}{2\pi i} e^{z(\phi)t} \widehat{u}(z(\phi)) z'(\phi).
	\end{equation}
	The integral (\ref{eq2.4}) can then be efficiently approximated using either the trapezoidal rule or the midpoint rule  \cite{martensen1968numerischen,stenger2012numerical}.
	
	In this paper, we adopt the midpoint rules to approximate the integral (\ref{eq2.4}), i.e.,
	\begin{equation} \label{eq2.5}
		\begin{aligned}
			u(t) &\approx I_\tau := \tau \sum_{k=-\infty}^{+\infty} g(\phi_k,t)
			\approx I_{\tau,N} :=\tau \sum_{k=1-N}^{N-1} g(\phi_k,t),
		\end{aligned}
	\end{equation}
	where $\tau$ is a uniform time step size along the $\phi$-axis, $N$ denotes the number of quadrature nodes, and $\phi_k := (k + \frac{1}{2}) \tau, k = 0, \pm 1, \ldots, \pm (N-1)$.
	
	The scheme (\ref{eq2.5}) is referred to the temporal CIM scheme for approximating \(u(t)\), which attains spectral accuracy under appropriate conditions. As shown in \cite{weideman2007parabolic}, the CIM imposes low regularity requirements on \(u(t)\), and it only necessitates the existence of the Laplace transform. In fact, the convergence behavior of (\ref{eq2.5}) is closely linked to the analyticity of \(\widehat{u}(z)\): the width of the region of analytic continuation  of \(\widehat{u}(z)\) around the contour \(\Gamma\) is positively correlated with the exponential convergence rate of the CIM error.
	
	The error of the CIM scheme (\ref{eq2.5}) can be decomposed into the discretization error $DE := \|I-I_\tau\|_{L^2(\Omega)}$ and truncation error $TE := \|I_\tau-I_{\tau;N}\|_{L^2(\Omega)}$,
	$$
	\|I-I_{\tau,N}\|_{L^2(\Omega)}\leq \|I-I_\tau \|_{L^2(\Omega)}+\|I_\tau-I_{\tau;N} \|_{L^2(\Omega)}=:DE+TE,
	$$
	where the definitions of $I$, $I_\tau$ and $I_{\tau;N}$ are given by (\ref{eq2.4}) and (\ref{eq2.5}), respectively. 
	In \cite{weideman2007parabolic}, the authors analyzed these two types of errors by considering the function \( g(\phi,t) \) in (\ref{eq2.4}) as complex-valued, i.e., $g(\nu+ic,t)$, where \( g(\phi,t) \) admits an analytic continuation w.r.t. $\phi=\nu+ic$, and the time $t$ can be regarded as a given parameter. They assume that
	$$
	\int_{-\infty}^{+\infty}\|g(\nu+ic,t)\|_{L^2(\Omega)} \,d\nu= M(c,t)<\infty,\quad c \in (-\tilde{c},\tilde{c}).
	$$
	which leads to the following estimate for the discretization error:
	\begin{equation}\label{DE}
		DE\leq \frac{M(\tilde{c},t)}{e^{2\pi \tilde{c}/\tau}-1},
	\end{equation}
	implying that $DE=O(e^{-2\pi \tilde{c}/\tau})$. Estimate (\ref{DE}) indicates that larger values of $\tilde{c}$, which represents the width of the region of analytic continuation of $\widehat{u}(z(\phi))$, result in smaller discretization errors and faster convergence.
	As for the truncation error $TE$, assuming that $g(\phi,t)$ decays rapidly as $\phi = \nu+ic$, with $\nu \to \pm \infty$, it can be approximated by the magnitude of the last retained term. That is, for fixed $\tau$ and as $N\to \infty$,
	\begin{equation}\label{TE}
		TE=O(|g(\tau N,t)|).
	\end{equation}
	
	\subsection{Temporal Semi-discrete CIM Scheme}\label{section-2.2}
	
	According to the key idea of the CIM, we begin by deriving the Laplace transform of \( u(x,t) \), hereafter abbreviated as \( u(t) \), from problem (\ref{eq1.1}). Since the Laplace transform of the Caputo fractional derivative (\ref{eq1.2}) with \( \beta \in (0, 1) \) is given by \cite{Podlubny:1999}
	$$
	\mathcal{L}\big\{{_0^C}\!D_t^{\beta} u(t)\big\} = z^{\beta} \widehat{u}(z) - z^{\beta-1} u(0),
	$$
	and the Laplace transform of the integral term is
	$$
	\mathcal{L}\Big\{\int_0^t \kappa_\alpha(t-s)Au(s)\,ds\Big\}=\frac {1}{z^{\alpha}}A\widehat{u}(z),
	$$
	then applying the Laplace transform to problem (\ref{eq1.1}) yields
	$$
	z^\beta \widehat{u}(z)-z^{\beta-1}u_0+A\widehat{u}(z)+\frac {1}{z^{\alpha}}A\widehat{u}(z)=\widehat{f}(z).
	$$
	It follows that
	\begin{equation} \label{eq2.9}
		\widehat{u}(z) = \big( m(z)I + A\big)^{-1} \Big( \frac{m(z)}{z} u_0 +\frac{m(z)}{z^\beta} \widehat{f}(z) \Big),
	\end{equation}
	where
	\begin{equation} \label{eq2.10}
		m(z) = \frac{z^{\alpha + \beta}}{z^\alpha + 1}.
	\end{equation}
	By appropriately choosing the range of values for \(z\), it is possible to ensure that \(( m(z)I + A)^{-1}\) remains free of singularities. A detailed discussion will be provided in Section \ref{section3}. We can then observe that the singularity of \( \widehat{u}(z) \) only occurs at \( z = 0 \), since \( z^{\alpha} = -1 \) does not hold when $z$ is within a single-valued branch.
	
	Secondly, an appropriate integral contour \(\Gamma\) must be chosen. Here, we adopt the left branch of a hyperbolic contour parameterized by \(\mu > 0\) and \(\theta > 0\), which is defined as
	\begin{equation} \label{eq2.11}
		\Gamma:	z(\phi) = \mu\big(1 + \sin(i\phi - \theta)\big), \quad \phi = \nu + ic \in K, \quad -\infty < \nu < \infty,
	\end{equation}
	where \(K\) denotes an open strip containing all admissible horizontal lines, given by
	$$
	K = \big\{ \phi = \nu + ic \in \mathbb{C} ~|~ \nu \in \mathbb{R},~|c| < \tilde{c} \big\},
	$$
	with $\tilde{c}$ to be determined.
	One reason for choosing a hyperbolic contour is that, as observed in \cite{weideman2007parabolic}, such contours generally require fewer discretization points \(N\) than other previously mentioned types to achieve the same level of accuracy. Another reason is that hyperbolas possess asymptotes, which are necessary for determining the optimal value of $\tilde{c}$. Further details on this matter will also be provided in Section \ref{section3}.
	
	To facilitate analysis and discussion, we express (\ref{eq2.11}) in hyperbolic form:
	\begin{equation} \label{eq2.12}
		\Big(\frac{a - \mu}{\sin(\theta + c)}\Big)^2 - \Big(\frac{b}{\cos(\theta + c)}\Big)^2 = \mu^2, \quad z = a + ib.
	\end{equation}
	Figure \ref{fig.1} provides a visual representation of the mapping between the horizontal region \( K \) in the \(\phi\)-plane and its corresponding region \( S \) in the \( z \)-plane. Actually, from (\ref{eq2.12}), we observe that when \( c = 0 \), the contour $\Gamma$ corresponds to the left branch of a hyperbola, shown in white on the right side of Figure  \ref{fig.1}. As \( c \) decreases from 0 (i.e., \( c < 0 \)), indicating a downward shift of the horizontal line in region \( K \), the corresponding hyperbola in region \( S \) expands and shifts rightward. At the minimum value \( c = -\theta \), its image in the \( z \)-plane becomes a vertical line. Conversely, when \( c > 0 \) and increases from 0, corresponding to an upward movement of the line in the \(\phi\)-plane, the hyperbola  in the $z$-plane gradually closes and shifts leftward, eventually coinciding with the negative real axis when \( c \) reaches \( \pi/2 - \theta \). 
	
	\begin{figure}[htbp]
		\centering
		\begin{minipage}[t]{0.45\textwidth}
			\centering
			\includegraphics[width=1\textwidth]{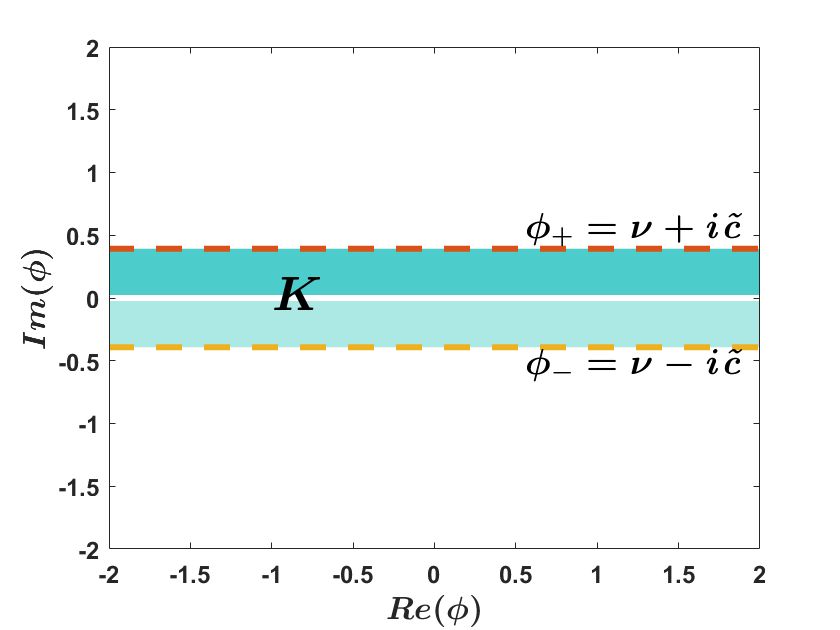}
		\end{minipage}
		\begin{minipage}[t]{0.45\textwidth}
			\centering
			\includegraphics[width=1\textwidth]{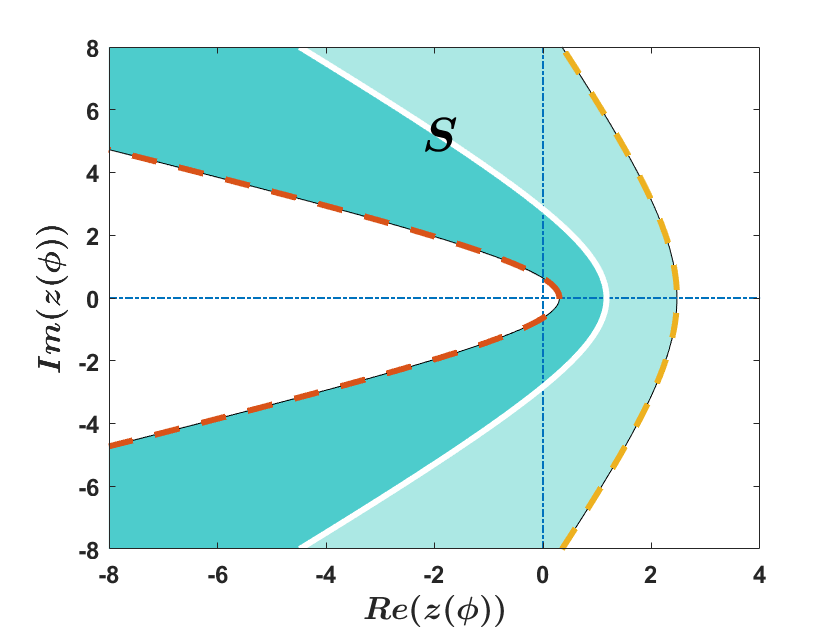}
		\end{minipage}
		\caption{A schematic mapping relationship between the horizontal region \( K \) in the \(\phi\)-plane and the region \( S \) in the \( z \)-plane.}
		\label{fig.1}
	\end{figure}
	
	It is essential that all singularities of the integrand (or \( \widehat{u}(z) \)) should lie to the left of the contour \( \Gamma \). As previously noted, for the problem we considered in this paper, the singularity is located only at the origin \( z = 0 \). Therefore, as shown in Figure \ref{contour}, the origin point must lie strictly to the left of \( \Gamma \).
	
	
	Finally, based on the CIM and leveraging the fact that the chosen contour \(\Gamma\) is symmetric with respect to the real axis (i.e., \(\widehat{u}(\overline{z}) = \overline{\widehat{u}(z)}\)), the temporal semi-discrete scheme for problem (\ref{eq1.1}) can be derived as:
	\begin{equation} \label{eq2.13}
		u^N(t) := I_{\tau;N} = \frac{\tau}{\pi} \operatorname{Im} \Big\{ \sum_{k=0}^{N-1} e^{z(\phi_k)t} \widehat{u}(z(\phi_k)) z'(\phi_k) \Big\},
	\end{equation}
	where \(\phi_k := (k + \frac{1}{2}) \tau\) for \(k = 0, 1, \ldots, N-1\). Here, \(\widehat{u}(z(\phi_k))\) and \(z(\phi_k)\) denote the values of (\ref{eq2.9}) and (\ref{eq2.11}) at \(\phi = \phi_k\), respectively. One can observe that compared to (\ref{eq2.5}), the scheme (\ref{eq2.13}) reduces the computational cost by half. For notational convenience, we denote \(z(\phi_k)\) as \(z_k\) in the subsequent analysis.
	
	\subsection{Spatial Semi-discrete Galerkin Scheme}
	
	We consider a family of quasi-uniform triangulations ${\{\mathcal{T}_h\}}_{0<h<1}$ of the domain $\bar{\Omega}$, with $\bar{h}$ denoting the maximum diameter among all finite elements, and construct the piecewise linear finite element function space $X_h\subset H_0^1(\Omega)$ over $\mathcal{T}_h$ consisting of continuous, piecewise linear functions:
	$$
	X_h := \big\{ \chi \in H_0^1(\Omega) : \chi|_{\tau} \text{ is a linear polynomial over } \tau \in \mathcal{T}_h \big\}.
	$$
	The semi-discrete Galerkin finite element approximation to problem (\ref{eq1.1}) is then formulated as follows: find $u_h(t)\in X_h$ such that for $0<t\leq T,$
	\begin{equation} \label{eq2.14}
		\left\{
		\begin{aligned}
			&\big({_0^C}\!D_{t}^{\beta}u_h(t),v_h\big) + \big(\nabla u_{h}(t), \nabla v_{h}\big)+\int_{0}^{t} \kappa_{\alpha}(t-s) \big(\nabla u_{h}(s), \nabla v_{h}\big)\,ds= (f,v_h)\quad \forall v_h\in X_h,\\
			&u_h(0)=u_{0,h}.
		\end{aligned}\right.
	\end{equation}
	Here $u_{0,h}:=P_hu_0$ denotes the $L^2$-projections of $u_{0}$ on $X_h$, where the $L^2$-projection operator $P_h$ is defined by
	$$
	P_h:L^2(\Omega)\to X_h,\,(P_hu,v_h)=(u,v_h)\quad \forall v_h\in X_h, u\in L^2(\Omega).
	$$
	
	For convenience, we rewrite the semidiscrete scheme (\ref{eq2.14}) in weak form as
	\begin{equation} \label{eq2.15}
		{_0^C}\!D_{t}^{\beta}u_h + A_hu_h+\int_{0}^{t} \kappa_{\alpha}(t-s) A_h u_h(s)\,ds= f_h \quad \forall t>0,
	\end{equation}
	where $f_h(t):=P_hf(t)$ and $A_h$ is the Ritz projection operator \cite{brenner2008mathematical} of $A$. After applying the Laplace transform to both sides of (\ref{eq2.15}), we have
	\begin{equation} \label{eq2.16}
		(z^\beta I+(1+z^{-\alpha})A_h)\widehat{u}_h(z)=z^{\beta-1}u_{0,h}+\widehat{f}_h(z).
	\end{equation}
	Consequently, the semi-discrete solution in space for problem (\ref{eq1.1}) can be expressed as
	\begin{equation} \label{eq2.17}
		u_h(t)=\frac{1}{2\pi i} \int_\Gamma e^{zt} (m(z)I + A_h)^{-1} \Big( \frac{m(z)}{z} u_{0,h} + \frac{m(z)}{z^\beta} \widehat{f}_h(z) \Big)dz,
	\end{equation}
	where $m(z)$ is given in (\ref{eq2.10}).
	
	\subsection{Fully Discrete Scheme and Acceleration of Algorithm}
	\label{subsection2.4}
	
	Applying the CIM to the spatial semi-discrete scheme (\ref{eq2.17}), we can obtain the fully discrete scheme for problem (\ref{eq1.1}) as
	\begin{equation} \label{eq2.18}
		u_h^N(t) := I_{\tau;h;N} = \frac{\tau}{\pi} \text{Im} \Big\{ \sum_{k=0}^{N-1} e^{z_k t}  \widehat{u}_h(z_k) z_k' \Big\},
	\end{equation}
	where $\widehat{u}_h(z_k)$ is obtained by solving (\ref{eq2.16}) at $z=z_k, k=0,1,2,...,N-1$.
	
	From (\ref{eq2.18}), it is evident that the solution at a given time \( t \) can be computed without relying on any information from previous time. This approach, referred to as a localization method, provides considerable computational advantages over nonlocal schemes.
	
	However, achieving optimal error convergence with the scheme (\ref{eq2.18}) typically requires $N$ to be between 100 and 200, which involves solving 100 to 200 elliptic equations (\ref{eq2.16}). The computational cost remains substantial as the spatial dimension increases. To further reduce the computational cost, one may employ the barycentric Lagrange interpolation \cite{berrut2004barycentric} to approximate all values of $\widehat{u}_h(z_k)$, leveraging the independence among these computations. Specifically, we first select $n+1$ Chebyshev interpolation nodes $z_{x_j}$ for $j=0,1,\ldots, n,$ along the truncated horizontal line. Then the values of $\widehat{u}_h(z_{x_j})$ can be computed by solving the associated elliptic equations  (\ref{eq2.16}). Finally, the values of $\widehat{u}_h(z_k)$ at the original equispaced nodes are approximated by using of the barycentric interpolation formula:
	$$
	\widehat{u}_h(z_k) \approx \widehat{u}_{I,h}^n(z_k) := \frac{\sum_{j=0}^{n} \frac{\omega_j}{z_k - z_{x_j}} \widehat{u}_h(z_{x_j})}{\sum_{j=0}^{n} \frac{\omega_j}{z_k - z_{x_j}}},
	$$
	where the barycentric weights $\omega_j$ are given by
	$$
	\omega_j := \frac{1}{\prod_{k \neq x_j} (z_{x_j} - z_k)}, \quad j = 0, 1, \ldots, n.
	$$
	In this manner, only $n+1$ elliptic equations need to be solved. Empirical results suggest that $n$ typically lies between $10$ and $20$, thereby enabling a further notable reduction in computational load. In view of the focus of this paper on localization and nonsmoothness, acceleration algorithms will not be further discussed in subsequent sections. For further details on the implementation and error analysis of this accelerated algorithm, we refer the reader to \cite[Section 6]{ma2023analyses}.
	
	\section{Solution Theory and Error Analysis}
	\label{section3}
	
	This section outlines the solution theory for problem (\ref{eq1.1}) (cf. \cite{mahata2023nonsmooth}), provides a proof of the spectral accuracy achieved by the CIM, and presents error estimates for both homogeneous and nonhomogeneous cases.
	
	\subsection{Solution Theory}\label{subsec:3.1}
	
	To begin, for $\vartheta\in(0,\pi]$, we introduce the sector
	$$
	\Sigma_{\vartheta} := \{ z \in \mathbb{C} : |\arg(z)| < \vartheta \},
	$$
	and assume that $\widehat{f}(z)$ is analytic and bounded on $S\subseteq\Sigma_{\tilde{\vartheta}}$ (the domain $S$ is shown in Figure \ref{fig.1}) for some $\tilde{\vartheta}\in(\pi/2,\pi)$, the associated norm is defined as
	\begin{equation}\label{eq:normS}
		\|\widehat{f}(z)\|_{S} := \sup_{z \in S} \|\widehat{f}(z)\|_{L^2(\Omega)}, \quad S \subseteq \Sigma_{\tilde{\vartheta}}.
	\end{equation}
	Naturally, the hyperbolic contour \(\Gamma\) discussed in Section \ref{section2} must lie entirely within the sector $\Sigma_{\vartheta}$, i.e., if $\delta<\frac{\pi }{2}$ denotes the angle between the asymptote of $\Gamma$ and the real axis, then $\delta \geq \pi-\vartheta$.
	
	In \cite{mahata2023nonsmooth}, Mahata and Sinha provided some important regularity properties for the solution $u(t)$ as follows:
	
	\begin{lemma}[Solution Theory]
		Let $j\in\{0,1\}$, $m\in\mathbb{N}$, and denote $u^{(m)}(t) := \frac{\partial^m u(t)}{\partial t^m}, f^{(m)}(t) := \frac{\partial^m f(t)}{\partial t^m}$. The solution $u(t)$ of problem (\ref{eq1.1}) satisfies the following properties.
		
		\begin{enumerate}[label=\roman*)]
			\item (Homogeneous Problem)
			If $u_0\in L^2(\Omega)$ and $f=0$, then the weak solution $u(t)$ is unique, the following a priori estimates hold for positive time:
			$$
			\|A^j u^{(m)}(t)\|_{L^2(\Omega)} \leq C t^{-(m+j\beta)} \|u_0\|_{L^2(\Omega)}.
			$$
			Moreover, we have
			$$
			\|u^{(m)}(t)\|_{L^p(\Omega)} \leq C t^{-(m + \frac{(p-q)\beta}{2})} \|u_0\|_{L^q(\Omega)}, \quad 0 \leq p, q \leq 2.
			$$
			\item (Nonhomogeneous Problem)
			If $u_0=0$ and $f \in C^m([0, T]; L^2(\Omega))$, the weak solution $u(t)$ is unique, the following prior estimates hold for positive time:
			$$
			\begin{aligned}
				&\|A u^{(m)}(t) \|_{L^2(\Omega)}\\
				\leq &C \sum_{j=0}^{m} t^{-j} \| f^{(m-j)}(0) \|_{L^2(\Omega)} +
				C \int_{0}^{t} (t-s)^{\beta-1} \| f^{(m)}(s) \|_{L^2(\Omega)}\,ds + \int_{0}^{t} \| f^{(m+1)}(s) \|_{L^2(\Omega)}\,ds.
			\end{aligned}
			$$
		\end{enumerate}
		\label{solution theory}
	\end{lemma}
	
	From Lemma \ref{solution theory}, one can see that the regularity of the solution at the initial moment will influence its behavior throughout the temporal evolution. Subsequent numerical experiments confirm that our method maintains spectral accuracy even under these conditions.
	
	Additionally, it is known from \cite{kato1961fractional} that the resolvent of $A=-\Delta$ satisfies the following estimate for $\vartheta\in(0,\pi/2)$ :
	$$
	\|(zI + A)^{-1}\|_{L^2(\Omega)} \leq C|z|^{-1}, \quad z \in \Sigma_{\pi - \vartheta}.
	$$
	
	The bound of $m(z)$ (see (\ref{eq2.10})) is an important tool in error analysis. In Lemma 2.2 of \cite{mahata2023nonsmooth}, Mahata and Sinha established an estimate for $|m(z)|$ under the condition $\alpha+\beta \geq 1$. This paper extends their result to a broader range $\alpha + \beta \in (\varepsilon,2-\varepsilon)$ with any  $0<\varepsilon < 1$, leading to the following result.
	
	\begin{lemma}\label{lemma1}
		Let $m(z)$ be defined as in (\ref{eq2.10}) with $\alpha + \beta \in (\varepsilon,2-\varepsilon)$ for any fixed $0<\varepsilon < 1$, and $\tilde{\vartheta}:= \frac{(\alpha+\beta+\varepsilon)\pi}{2(\alpha + \beta)}$. Then for any $ z \in \Sigma_{\tilde{\vartheta}}$, there hold
		
		\begin{enumerate}[label=\alph*)]
			\item $m(z)\in \Sigma_\zeta$ with $\zeta = \frac{(\alpha+\beta+\varepsilon)\pi}{2}$,
			
			\item $|m(z)| \leq C |z|^{\alpha + \beta}$.
		\end{enumerate}  	
		
		Consequently, it leads to the following resolvent estimate:
		\begin{equation} \label{prior}
			\| (m(z)I + A)^{-1} \|_{L^2(\Omega)} \leq C |m(z)|^{-1}  \quad \forall z \in \Sigma_{\tilde{\vartheta}}.
		\end{equation}
	\end{lemma}
	
	\begin{proof}
		Since the range of $\alpha+\beta$ only affects the domain of $m(z)$, it suffices to prove part $a)$ here; the conclusion of part $b)$ then follows directly from \cite{Mahata2021FiniteEM}.
		Let $z =re^{i\psi} \in \Sigma_{\tilde{\vartheta}}$. Then $|\psi|< \tilde{\vartheta} = \frac{(\alpha+\beta+\varepsilon)\pi}{2(\alpha + \beta)}$ and $r>0$. Now,
		$$
		m(z)=\frac{z^{\alpha+\beta}}{z^\alpha+1}=\frac{r^{2\alpha+\beta}e^{i\beta\psi}+r^{\alpha+\beta}e^{i(\alpha+\beta)\psi}}{1+2r^\alpha \cos(\alpha \psi)+r^{2\alpha}} := m_1(z)+m_2(z),
		$$
		where
		$$
		m_1(z) := \frac{r^{2\alpha+\beta}e^{i\beta\psi}}{1+2r^\alpha \cos(\alpha \psi)+r^{2\alpha}}, ~~m_2(z) := \frac{r^{\alpha+\beta}e^{i(\alpha+\beta)\psi}}{1+2r^\alpha \cos(\alpha \psi)+r^{2\alpha}}.
		$$
		
		For $m_1(z)$, the argument satisfies
		$$
		|\arg(m_1)| = \beta|\psi|<\frac{\beta}{\alpha + \beta} \cdot \frac{(\alpha+\beta+\varepsilon)\pi}{2}<\frac{(\alpha+\beta+\varepsilon)\pi}{2},
		$$
		and for $m_2(z)$,
		$$
		|\arg(m_2)| = (\alpha+\beta)|\psi|<\frac{(\alpha+\beta+\varepsilon)\pi}{2}.
		$$
		Hence, $m(z) \in \Sigma_\zeta$ with $\zeta=\frac{(\alpha+\beta+\varepsilon)\pi}{2}$.
	\end{proof}
	
	\begin{remark}
		The result in Lemma \ref{lemma1} implies that \(m(z)\) can not be negative real number. As a result, \((m(z)I + A)^{-1}\) is well-defined, and \(\widehat{u}(z)\) in \eqref{eq2.9} only has a singular point \(z = 0\).
	\end{remark}
	
	\begin{remark}
		Now, the determination of parameter \( \tilde{c} \), a question left open in Section 2, can also be addressed. As indicated by (\ref{eq2.12}), the slope of the asymptote of the hyperbolic contour \( \Gamma \) is given by $\frac{\sin(\theta + c)}{\cos(\theta + c)} = \cot(\theta + c),$
		that is, $
		\tan \delta = \cot(\theta + c),$
		where \( \delta \), as introduced earlier, denotes the angle between this asymptote and the real axis. It follows that
		$$
		\delta = \frac{\pi}{2} - \theta - c \geq \pi - \tilde{\vartheta}.
		$$
		That is,
		$$
		c\leq \tilde{\vartheta}-\frac{\pi}{2}- \theta=\frac{\varepsilon \pi}{2(\alpha + \beta)}-\theta.
		$$
		In fact, since $\alpha + \beta \in (\varepsilon,2-\varepsilon)$, then $\frac{\varepsilon \pi}{2(\alpha + \beta)}-\theta<\frac{\pi}{2}-\theta$.
		Thus, based on the discussions in Subsection \ref{section-2.2}, it follows
		$$
		\tilde{c}=\min\Big\{\theta,\frac{\varepsilon \pi}{2(\alpha + \beta)}-\theta\Big\}.
		$$
    \end{remark}
	
	We note that all subsequent analyses are performed within the sector $\Sigma_{\tilde{\vartheta}}$, where $\tilde{\vartheta}$ is given in Lemma \ref{lemma1}.
	
	\subsection{Error Analysis in Temporal Semi-discrete CIM Scheme}

	In this subsection, we analyze the error estimate for the temporal semi-discrete CIM scheme \eqref{eq2.13}. First, it is essential to derive the decay property of $g(\phi,t)$ in (\ref{eq.2.5}) as presented in the following theorem.
	
	\begin{theorem} \label{Terr1}
		Let $g(\phi,t)$ be defined as in (\ref{eq.2.5}), where \( z(\phi) \in S \subseteq \Sigma_{\tilde{\vartheta}} \), $\alpha + \beta \in (\varepsilon,2-\varepsilon)$ with $0<\varepsilon < 1$, and \( \phi = \nu+ic \in K \), $|c|<\tilde{c}$. For \( t > 0 \), if \( u_0 \in L^2(\Omega) \), and $\|\widehat{f}(z)\|_{S}<\infty$ defined in \eqref{eq:normS}, then there exists a constant \( C>0 \) such that
		\begin{equation} \label{eqthm1}
			\|g(\phi,t)\|_{L^2(\Omega)} \leq C \varphi(\theta,\tilde{c}) e^{\mu t} \big( \| u_0 \|_{L^2(\Omega)} + \| \widehat{f}(z) \|_{S} \big)e^{-\mu t \sin(\theta-\tilde{c})\cosh \nu},
		\end{equation}	
		where $\varphi(\theta,\tilde{c}):=\sqrt{\frac{1+\sin(\theta+\tilde{c})}{1-\sin(\theta+\tilde{c})}}.$
	\end{theorem}
	
	\begin{proof}
		Firstly, we estimate $g(\phi,t)$ on the upper part of the domain $K$: $\phi_+=\nu+ic\in K$, $0<c<\tilde{c}$. From (\ref{eq.2.5}) and (\ref{eq2.11}), we obtain
		$$
		g(\nu+ic,t)=\frac{\mu}{2\pi}e^{\mu t\{1-\sin(\theta+c-i\nu)\}}\cos(\theta+c-i\nu) \widehat{u}\big(\mu\{1-\sin(\theta +c-i\nu)\}\big).
		$$
		Let $\varpi=\theta+c<\theta+\tilde{c}$. Then the following estimate holds:
		\begin{equation} \label{eq3.23}
			\|g(\nu+ic,t)\|_{L^2(\Omega)}\leq Ce^{\mu t(1-\sin \varpi\cosh \nu)}|\cos(\varpi-i\nu)| \cdot \big\|\widehat{u}\big(\mu\{1-\sin(\varpi-i\nu)\}\big)\big\|_{L^2(\Omega)}.
		\end{equation}
		The main task now is to estimate \( \widehat{u}(z) \). Using the expression of \( \widehat{u}(z) \) in (\ref{eq2.9}) and the prior estimate (\ref{prior}), it follows that for all $z\in \Sigma_{\tilde{\vartheta}}$ with $\alpha + \beta \in (\varepsilon,2-\varepsilon)$ and $0<\varepsilon < 1$,
		$$
		\begin{aligned}			
			\|\widehat{u}(z)\|_{L^2(\Omega)}
			&=\Big\|(m(z)I+A)^{-1}\Big(\frac{m(z)}{z}u_0+\frac{m(z)}{z^\beta}\widehat{f}(z)\Big)\Big\|_{L^2(\Omega)}\\
			&\leq C|z|^{-1}\|u_0\|_{L^2(\Omega)}+C|z|^{-\beta}\|\widehat{f}(z)\|_{L^2(\Omega)}.
		\end{aligned}
		$$
		Thus,
		\begin{equation} \label{eq3.25}
			\begin{aligned}
				&|\cos(\varpi-i\nu)| \cdot \big\|\widehat{u}\big(\mu\{1-\sin(\varpi-i\nu)\}\big)\big\|_{L^2(\Omega)}\\
				&\leq C \frac{|\cos(\varpi-i\nu)|}{|\mu(1-\sin(\varpi-i\nu))|}\|u_0\|_{L^2(\Omega)}
				+C\frac{|\cos(\varpi-i\nu)|}{|\mu(1-\sin(\varpi-i\nu))|^\beta}\|\widehat{f}(z)\|_{L^2(\Omega)}\\
				&\leq C \sqrt{\frac{1+\sin \varpi}{1-\sin \varpi}}\|u_0\|_{L^2(\Omega)}+C(|1-\sin(\varpi-i\nu)|^{1-\beta})
				\sqrt{\frac{1+\sin \varpi}{1-\sin \varpi}}\|\widehat{f}(z)\|_{L^2(\Omega)},
			\end{aligned}
		\end{equation}
		where in the second inequality of (\ref{eq3.25}), we employ the following result from \cite{LopezFernndez2004OnTN}:
		$$
		\Big|\frac{\cos(\varpi - i\nu)}{(1 - \sin(\varpi - i\nu))}\Big|^2 = \frac{\cosh \nu+\sin \varpi}{\cosh \nu-\sin \varpi}
		\leq \frac{1 + \sin \varpi}{1 - \sin \varpi} .
		$$
		Given that $0<\beta<1$, it follows that $|1-\sin(\varpi-i\nu)|^{1-\beta}\leq 2^{1-\beta}$. Then according to (\ref{eq3.23}) and (\ref{eq3.25}), there holds
		$$
		\|g(\nu+ic,t)\|_{L^2(\Omega)} \leq C \sqrt{\frac{1+\sin(\theta+\tilde{c})}{1-\sin(\theta+\tilde{c})}} e^{\mu t} \big( \| u_0 \|_{L^2(\Omega)} + \| \widehat{f}(z) \|_{S} \big)e^{-\mu t \sin(\theta-\tilde{c})\cosh \nu}.
		$$
		We next estimate $g(\phi,t)$ on the lower half of $K$: $\phi_-=\nu-ic\in K$, $-\tilde{c}<-c\leq 0$. Since $0<\theta-\tilde{c}<\theta+\tilde{c}<\frac{\pi}{2}$ and $\sin(\theta-\tilde{c})<\sin\theta<\sin(\theta+\tilde{c})$, an analogous argument yields the same estimate:
		$$
		\|g(\nu-ic,t)\|_{L^2(\Omega)} \leq C \sqrt{\frac{1+\sin(\theta+\tilde{c})}{1-\sin(\theta+\tilde{c})}} e^{\mu t} \big( \| u_0 \|_{L^2(\Omega)} + \| \widehat{f}(z) \|_{S} \big)e^{-\mu t \sin(\theta-\tilde{c})\cosh \nu}.
		$$
		In summary, for all $\phi=\nu+ic\in K$, the function $g(\phi,t)$ satisfies estimate (\ref{eqthm1}). This completes the proof.
	\end{proof}
	
	From Theorem \ref{Terr1}, it can be observed that $g(\nu+ic,t)$ exhibits biexponential decay w.r.t $\nu$, that is, it decays quite rapidly as $\phi = \nu+ic$, $\nu \to \pm \infty$. To further analyze the error of the temporal semi-discrete CIM scheme \eqref{eq2.13}, we need to introduce the following lemma.
	
	\begin{lemma}[cf. Lemma 1 in \cite{LopezFernndez2004OnTN}]
		\label{Tl1}
		Let $L(x)=1+|\ln(1-e^{-x})|,\,x>0$. There hold
		$$
		\int_{0}^{+\infty} e^{-\gamma \cosh x} \, dx \leq L(\gamma), \quad \gamma > 0,
		$$
		and
		$$
		\int_{l}^{+\infty} e^{-\gamma \cosh x} \, dx \leq \big(1 + L(\gamma)\big) e^{-\gamma \cosh l}, \quad l > 0,~ \gamma > 0.
		$$
	\end{lemma}
	
	As noted in \cite{weideman2007parabolic}, each new value of \( t \) leads to different parameters \( \theta \) and \( \tilde{c} \), implying that the transform \( \widehat{u}(z) \) must be sampled on a distinct contour. This could result in computational inefficiency in the matrix problem associated with the discretization of the spatial operator \( A \) in (\ref{eq2.16}), as well as in the observation of the long-term dynamic evolution of the simulated solution over time. To address this issue and select a unified optimal contour \( \Gamma \), a scaling parameter \( \Lambda > 1 \) is introduced. The subsequent analysis is performed over the interval \( t \in [t_0, t_1] \), where \( t_1 = \Lambda t_0 \), and a uniform error analysis of the temporal semi-discrete CIM scheme \eqref{eq2.13} is presented in the following theorem. 
	
	\begin{theorem} \label{Terr2}
		Let $u(t)$ and $u^N(t)$ be the solutions given by (\ref{eq2.3}) and (\ref{eq2.13}), respectively. Consider the time interval $t \in [t_0, t_1]$ with $t_1 = \Lambda t_0$ for some $\Lambda \geq 1$. Let $\Gamma$ be the contour defined in (\ref{eq2.11}) with parameters $\theta>0, \mu >0$, and $z\in \Sigma_{\tilde{\vartheta}}$ with $\alpha + \beta \in (\varepsilon,2-\varepsilon)$, $0<\varepsilon < 1$. Then, the following error estimate holds uniformly:
		$$
		\|u(t)-u^N(t)\|_{L^2(\Omega)}\leq C\varphi(\theta,\tilde{c})L\big(\mu t_0 \sin(\theta-\tilde{c})\big) \big( \| u_0 \|_{L^2(\Omega)} + \| \widehat{f}(z) \|_{S} \big) \frac{e^{\mu \Lambda t_0}}{e^{2\pi\tilde{c}/\tau}-1},
		$$
		where $\tilde{\vartheta}$ is given in Lemma \ref{lemma1}, and $\varphi(\theta,\tilde{c})$ is defined in Theorem \ref{Terr1}. The parameters $\mu$ and $\tau$ are chosen according to \cite{lopez2006spectral} as
		\begin{equation} \label{thm2.2}
			\mu=\frac{2\pi\tilde{c}N(1-\eta)}{\Lambda t_0 P(\eta)},\quad\tau=\frac{P(\eta)}{N}
		\end{equation}
		for any $0<\eta<1$, and
		$$
		P(\eta):=\rm{arcosh}\Big(\frac{\Lambda}{(1-\eta)\sin(\theta-\tilde{c})}\Big).
		$$
	\end{theorem}
	
	\begin{proof}
		Applying the estimate of $\|g(\phi,t)\|_{L^2(\Omega)}$ given in Theorem \ref{Terr1}, we have the estimate of the truncation error $TE = \|I_\tau-I_{\tau;N}\|_{L^2(\Omega)}$ as follows
		$$
		\begin{aligned}
			TE&\le\tau \sum_{k=N}^{+\infty}\|g(\phi_k,t)\|_{L^2(\Omega)}\\
			&\leq \tau \sum_{k=N}^{\infty} \big(C\varphi(\theta,\tilde{c})\,e^{\mu t}(\|u_0\|_{L^2(\Omega)}+\|\widehat{f}(z)\|_{S})e^{-\mu t \sin(\theta-\tilde{c}) \cosh((k+\tfrac{1}{2})\tau)} \big)\\
			&\leq C\varphi(\theta,\tilde{c})\,e^{\mu t}\big(\|u_0\|_{L^2(\Omega)}+\|\widehat{f}(z)\|_{S}\big) \cdot \int_{N\tau}^{+\infty} e^{-\mu t\sin(\theta-\tilde{c})\cosh x} \, dx.
		\end{aligned}	
		$$
		By Lemma \ref{Tl1}, it follows
		$$
		\int_{N\tau}^{+\infty} e^{-\mu t\sin(\theta-\tilde{c})\cosh x} \, dx \leq \big(1 + L(\mu t \sin(\theta-\tilde{c}))\big) e^{-\mu t \sin(\theta-\tilde{c}) \cosh (N\tau)},
		$$
		which yields the estimate
		\begin{equation} \label{TE2}
			TE \leq C\varphi(\theta,\tilde{c}) L\big(\mu t \sin(\theta-\tilde{c})\big)  \big(\|u_0\|_{L^2(\Omega)}+\|\widehat{f}(z)\|_{S}\big) e^{\mu t} \cdot \frac{1}{e^{\mu t \sin(\theta-\tilde{c}) \cosh (N\tau)}}.
		\end{equation}
		
		As for the discretization error $DE = \|I-I_\tau\|_{L^2(\Omega)}$, it follows from the estimate (\ref{DE}) and Lemma \ref{Tl1} that
		\begin{equation} \label{DE1}
			\begin{aligned}
				DE &\leq \frac{1}{e^{2\pi \tilde{c}/\tau}-1}\int_{-\infty}^{+\infty}\|g(\nu+i\tilde{c},t)\|_{L^2(\Omega)}\,d \nu\\
				& \leq  C\varphi(\theta,\tilde{c})\,e^{\mu t}\big(\|u_0\|_{L^2(\Omega)}+\|\widehat{f}(z)\|_{S}\big) \Big(\int_{-\infty}^{+\infty} e^{-\mu t \sin(\theta-\tilde{c}) \cosh \nu}\,d\nu \Big)  \cdot \frac{1}{e^{2\pi \tilde{c}/\tau}-1}\\
				& = C\varphi(\theta,\tilde{c})\,e^{\mu t}\big(\|u_0\|_{L^2(\Omega)}+\|\widehat{f}(z)\|_{S}\big) \Big(2 \int_{0}^{+\infty} e^{-\mu t \sin(\theta-\tilde{c}) \cosh \nu}\,d\nu \Big)  \cdot \frac{1}{e^{2\pi \tilde{c}/\tau}-1}\\
				& \leq C\varphi(\theta,\tilde{c})\,L\big(\mu t \sin(\theta-\tilde{c})\big) \big(\|u_0\|_{L^2(\Omega)}+\|\widehat{f}(z)\|_{S}\big) e^{\mu t} \cdot \frac{1}{e^{2\pi \tilde{c}/\tau}-1}.
			\end{aligned}
		\end{equation}
		The estimates (\ref{TE2}) and (\ref{DE1}) indicate the asymptotic behavior
		\begin{equation} \label{TEDE1}
			TE = O(e^{-\mu t \sin(\theta-\tilde{c}) \cosh(\tau N)}), \quad DE = O(e^{-2\pi \tilde{c}/\tau}).
		\end{equation}
		It indicates in (\ref{TEDE1}) that the truncation error $TE$ decreases as time $t$ grows for $t \in [t_0,t_1]$, $t_1=\Lambda t_0$, and $\Lambda \geq 1$. Therefore, our strategy for selecting parameters is to equilibrate the truncation error at the initial time \( t = t_0 \) and the discretization error \( DE \). Specifically, we enforce the condition
		\begin{equation} \label{eq.balance}
			e^{-\mu t_0 \sin(\theta - \tilde{c}) \cosh(\tau N)} = e^{-2\pi \tilde{c}/\tau},
		\end{equation}
		while simultaneously minimizing \( DE \). Solving (\ref{eq.balance}) yields
		\begin{equation} \label{P1}
			\tau N = \rm{arcosh}\Big(\frac{2\pi\tilde{c}}{\mu \tau t_0 \sin(\theta-\tilde{c})} \Big).
		\end{equation}
		To determine the parameters $\mu$ and $\tau$, we introduce a free parameter $0<\eta<1$, and define the relation
		\begin{equation} \label{P2}
			\frac{2\pi\tilde{c}}{\mu \tau t_0} = \frac{\Lambda}{1-\eta}.
		\end{equation}
		Solving the system given by (\ref{P1}) and (\ref{P2}) yields the explicit expressions
		$$
		\tau = \frac{1}{N}{\rm{arcosh}}\Big(\frac{\Lambda}{(1-\eta) \sin(\theta-\tilde{c})} \Big) := \frac{P(\eta)}{N}, \quad \mu = \frac{2\pi\tilde{c} (1-\eta) N}{\Lambda t_0 P(\eta)}.
		$$
		
		Furthermore, by combining the estimates (\ref{TE2}) and (\ref{DE1}), we obtain the uniform error bound
		$$
		\begin{aligned}
			&\|u(t)-u^N(t)\|_{L^2(\Omega)}\leq DE + TE\\
			&\leq C\varphi(\theta,\tilde{c}) L\big(\mu t_0 \sin(\theta-\tilde{c})\big) \big( \| u_0 \|_{L^2(\Omega)} + \| \widehat{f}(z) \|_{S} \big) e^{\mu \Lambda t_0} \Big( \frac{1}{e^{2\pi \tilde{c}/\tau} - 1} + \frac{1}{e^{\mu t_0 \sin(\theta - \tilde{c}) \cosh(N \tau)}} \Big)\\
			&\leq C\varphi(\theta,\tilde{c}) L\big(\mu t_0 \sin(\theta-\tilde{c})\big) \big( \| u_0 \|_{L^2(\Omega)} + \| \widehat{f}(z) \|_{S} \big) \Big( \frac{e^{\mu \Lambda t_0}}{e^{2\pi \tilde{c}/\tau} - 1}\Big).
		\end{aligned}
		$$
		This completes the proof.
	\end{proof}
	
	In fact, as demonstrated in numerical results, with an appropriate parameter choice, the actual error of the temporal semi-discrete CIM scheme \eqref{eq2.13} exhibits the asymptotic behavior
	$$
	\|u(t)-u^N(t)\|_{L^2(\Omega)}=O(e^{\mu\Lambda t_0-2\pi\tilde{c}/\tau})=O(e^{-[2\pi\tilde{c}\eta/P(\eta)]N}).
	$$
	
	Several key observations can be made from this result:
	\begin{enumerate}
		\item
		Spectral Accuracy: Noting that $P(\eta) > \rm{arcosh}(\Lambda) = \ln(\Lambda + \sqrt{\Lambda^2-1}) \sim \ln(\Lambda)$, we confirm the method's spectral accuracy:
		\begin{equation} \label{accuracy}
			\|u(t) - u^N(t)\|_{L^2(\Omega)} = O(e^{-C N}) \quad \text{with} \quad C = O(1/\ln \Lambda).
		\end{equation}
		
		\item
		Analytic Bandwidth: The decay rate improves with a wider analytic bandwidth $\tilde{c}$ of the integrand $g(\phi,t)$ (or $\widehat{u}(z(\phi))$) in the $\phi$-plane. This bandwidth is determined solely by the regularity of the Laplace transform $\widehat{u}(z)$ in the $z$-plane, as $z(\phi)$ is a sufficiently smooth function. While the regularity of the solution $u(t)$ itself influences the magnitude of $\tilde{c}$ through the transformation, \textbf{our method guarantees an exponential convergence rate irrespective of the smoothness of $u(t)$}.
		
		\item
		Long-Time Simulation: (\ref{accuracy}) reveals a critical trade-off: a larger $\Lambda$ (i.e., a longer simulation time $t_1 = \Lambda t_0$) results in a smaller constant $C$, leading to a slower exponential decay. Consequently, achieving a prescribed accuracy for long-time simulations requires a significantly larger number of quadrature points $N$. This inherent characteristic of the CIM motivates the development of the acceleration algorithms discussed previously in Subsection \ref{subsection2.4}.
	\end{enumerate}
	
	\begin{remark}[Parameters Selection Procedure]
		For fixed values of $\theta$, $t_0$, $\Lambda$, and $N$, we can determine an optimal $\eta=\eta^*$ that minimizes the error estimate
		$$
		O(e^{-[2\pi\tilde{c}\eta/P(\eta)]N}):=O(e^{-Q(\eta)N})
		$$
		by maximizing the function $Q(\eta):=2\pi\tilde{c}\eta/P(\eta)$. Denote $Q(\eta*):=\max_{ 0<\eta<1}{Q(\eta)}$. Substituting $\eta^*$ into (\ref{thm2.2}) then yields the optimal parameters $\mu^*$ and $\tau^*$ for the contour $\Gamma$. The complete parameter selection procedure is presented as follows.
		
		\begin{algorithm}[H]\label{alg1}
			\SetAlgoLined
			\caption{Optimal parameters selection for $\Gamma$}
			\KwIn{$\theta$, $t_0$, $\Lambda \geq 1$, $N$.}
			\KwOut{The optimal parameters $\mu^*$ and $\tau^*$:}
			1. $P(\eta) \gets \rm{arcosh}\Big(\frac{\Lambda}{(1-\eta)\sin(\theta-\bar{c})}\Big)$\;
			2. $Q(\eta) \gets \frac{2\pi \bar{c} \eta}{P(\eta)}$\;
			3. $\eta^* \gets \max_{\eta \in (0,1)}Q(\eta)$\;	
			4. $\tau^*, \mu^* \gets$ Substitute $\eta^*$ into (\ref{thm2.2}) to get the optimal parameters $\tau^*$ and $\mu^*$.
		\end{algorithm}
	\end{remark}
	
	\subsection{Error Analysis of the Fully Discrete Scheme}
	
	The following two lemmas provide error estimates for the spatial semi-discrete scheme (\ref{eq2.17}), addressing the nonhomogeneous problem with zero initial data and the homogeneous problem with nonsmooth initial data, respectively. The proofs can be found in Theorems 3.1 and 3.2 of \cite{mahata2023nonsmooth}.
	
	\begin{lemma} \label{Serr1}
		Assume $f\equiv 0$ and $u_0 \in L^2(\Omega)$. Let $u(t)$ and $u_h(t)$ be the solutions given by (\ref{eq2.3}) and (\ref{eq2.17}), respectively. Then for $u_{0,h}=P_hu_0$, we have
		$$
		\|u(t) - u_h(t)\|_{L^2(\Omega)}+h\|u(t) - u_h(t)\|_{H_0^1(\Omega)} \leq Ch^2 t^{-\beta}\|u_0\|_{L^2(\Omega)}.
		$$
	\end{lemma}
	
	\begin{lemma} \label{Serr2}
		Assume $f\neq 0$ and $u_0\equiv 0$. Let $u(t)$ and $u_h(t)$ be the solutions given by (\ref{eq2.3}) and (\ref{eq2.17}), respectively. Then
		$$
		\|u(t) - u_h(t)\|_{L^2(\Omega)} +h\|u(t) - u_h(t)\|_{H_0^1(\Omega)}\leq Ch^2 \Big( \|f(0)\| + \int_0^t \|f'(s)\| ds \Big).
		$$
	\end{lemma}
	
	Now, by combining Theorem \ref{Terr2} and Lemma \ref{Serr1}, we can derive the error estimate for the fully discrete scheme in the case $u_0\in L^2(\Omega)$ and $f=0$.
	
	\begin{theorem} \label{thm3}
		Let $u(t)$ and $u_h^N(t)$ be the solutions given by (\ref{eq2.3}) and (\ref{eq2.18}), respectively. Consider the contour (\ref{eq2.11}), with $z\in \Sigma_{\tilde{\vartheta}}$, where $\alpha + \beta \in (\varepsilon,2-\varepsilon)$ for $0<\varepsilon < 1$, and let the parameters $\mu$ and $\tau$ be determined according to (\ref{thm2.2}). For $t\in[t_0,t_1]$, with $t_1=\Lambda t_0$ and $\Lambda \geq 1$, if $f\equiv 0$ and $u_0 \in L^2(\Omega)$, then the following error estimate holds:
		$$
		\|u(t)-u_h^N(t)\|_{L^2(\Omega)}
		\leq C \Big(h^2 t^{-\beta}+\varphi(\theta,\tilde{c})L\big(\mu t_0 \sin(\theta-\tilde{c})\big)\frac{e^{\mu \Lambda t_0}}{e^{2\pi\tilde{c}/\tau}-1} \Big)\|u_0\|_{L^2(\Omega)}.
		$$
	\end{theorem}
	
	Similarly, the error estimate for the fully discrete scheme with $u_0=0$ can be obtained by combining Theorem \ref{Terr2} and Lemma \ref{Serr2}.
	
	\begin{theorem} \label{thm4}
		Let $u(t)$ and $u_h^N(t)$ be the  solutions given by (\ref{eq2.3}) and (\ref{eq2.18}), respectively. Consider the contour (\ref{eq2.11}), with $z\in \Sigma_{\tilde{\vartheta}}$, where $\alpha + \beta \in (\varepsilon,2-\varepsilon)$ for $0<\varepsilon < 1$, and let the parameters $\mu$ and $\tau$ be determined according to (\ref{thm2.2}). For $t\in[t_0,t_1]$, with $t_1=\Lambda t_0$ and $\Lambda \geq 1$, if $f\neq 0$ and $u_0\equiv 0$, then the following error estimate holds:
		$$
		\begin{aligned}
			&\|u(t)-u_h^N(t)\|_{L^2(\Omega)}\\
			&\leq C \Big(h^2 \Big( \|f(0)\| + \int_0^t \|f'(s)\|ds \Big)+\varphi(\theta,\tilde{c})L\big(\mu t_0 \sin(\theta-\tilde{c})\big)\frac{e^{\mu \Lambda t_0}}{e^{2\pi\tilde{c}/\tau}-1}\|\widehat{f}(z)\|_S\Big).
		\end{aligned}		
		$$
	\end{theorem}
	
	\section{Numerical Results}
	\label{section4}
	
	In this section, we present four numerical examples to demonstrate the efficacy of the proposed spectral localization method for time-fractional integro-differential equations. The domain is set to \(\Omega = (0,1)\) or \((0,1) \times (0,1)\). For temporal discretization, we employ the CIM discussed in this paper. Specifically, for the hyperbolic contour (\ref{eq2.11}), we select the parameter \(\theta = 0.6767\) and set the initial time as \(t_0 = 0.1\). $T=\Lambda t_0$ for some given $\Lambda$. All the other parameters are determined in accordance with Algorithm \ref{alg1}. Spatial discretization is carried out using the standard piecewise linear Galerkin FEM.
	
	Given a fixed spatial mesh size $h$, the temporal errors are defined by
	$$
	\text{Err}_\tau(N):= \max_{t_0 \leq t \leq \Lambda t_0} \| u(\cdot,t) - u_h^N(\cdot,t) \|_{L^2(\Omega)},
	$$
	and for a fixed number of temporal nodes $N$, the spatial errors at time $t$ are measured by
	$$
	\text{Err}_h(t):= \| u(\cdot,t) - u_h^N(\cdot,t) \|_{L^2(\Omega)}.
	$$
	The spatial convergence order is computed using the formula:
	$$
	\text{Order} := \frac{\ln{\big(\mathrm{Err}_h(t)/\mathrm{Err}_{\frac{h}{2}}(t)\big)}}{\ln2}.
	$$
	All numerical examples were computed in MATLAB R2020a on a laptop with an Intel(R) Core(TM) i5-1035G1 CPU @ 1.00GHz.
	
	\subsection{Scalar Problem}
	
	\begin{example}
		Let $A=1$. We consider the scalar initial value problem
		\begin{equation} \label{ex1}
			\left\{
			\begin{aligned}
				&{_0^C}\!D_t^\beta u(t) + u(t) + \int_0^t \kappa_\alpha(t-s)u(s)\,ds = f(t), \quad t>0, \\
				&u(0) = 1,
			\end{aligned}\right.
		\end{equation}
		with the source term
		$$
		f(t) = 1 + \frac{3\sqrt{\pi}}{2}t + \frac{3\sqrt{\pi}}{2\Gamma(2-\beta)}t^{1-\beta} + \frac{1}{\Gamma(\alpha+1)}t^{\alpha} + \frac{3\sqrt{\pi}}{2\Gamma(\alpha+2)}t^{\alpha+1}.
		$$
		The exact solution is given by
		$$
		u(t) = 1 + \frac{3\sqrt{\pi}}{2}t.
		$$
	\end{example}
	
	To validate the numerical performance of the CIM, we conduct tests with different values of \(\alpha\), \(\beta\), and \(\Lambda\), as shown in Figure \ref{fig.2}.
	
	\begin{figure}[htbp]
		\centering
		\begin{subfigure}[c]{0.45\textwidth}
			\includegraphics[width=1\textwidth]{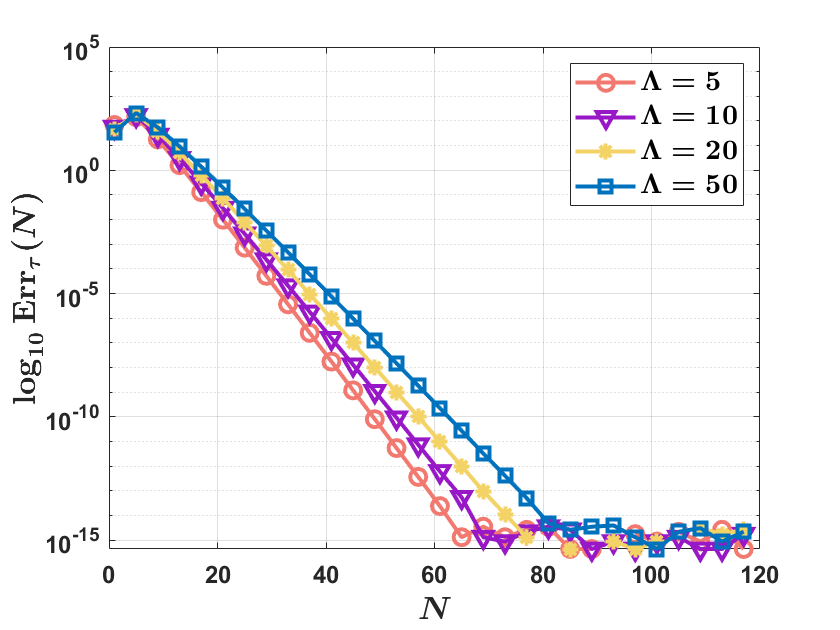}
			\caption{The absolute errors for different $\Lambda$ at $t=0.5$ with $\alpha=0.2$, $\beta=0.77$.}
			\label{fig:sub1}
		\end{subfigure}
		\hfill
		\begin{subfigure}[c]{0.45\textwidth}
			\includegraphics[width=1\textwidth]{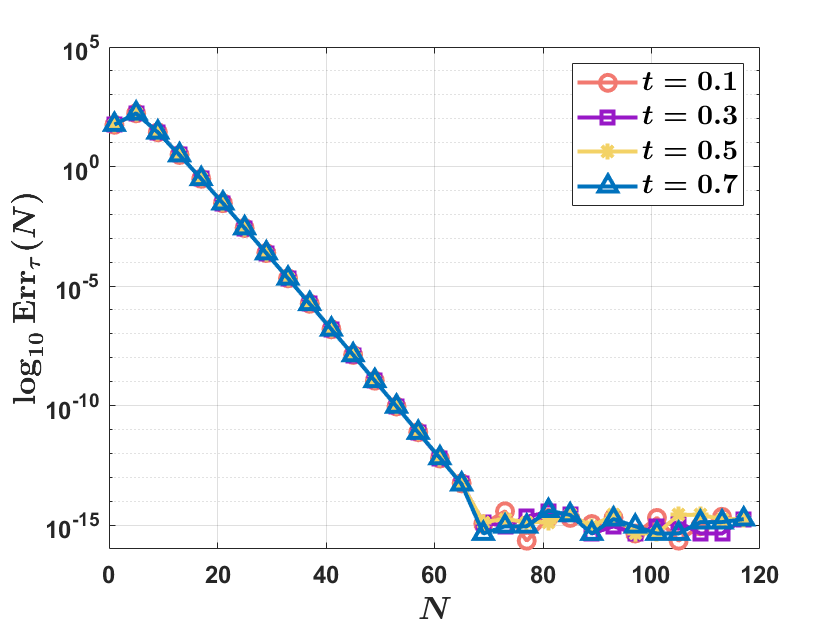}
			\caption{The absolute errors at different $t$ for $\Lambda=10$ with $\alpha=0.2$, $\beta=0.77$.}
			\label{fig:sub2}
		\end{subfigure}
		\hfill
		\begin{subfigure}[c]{0.45\textwidth}
			\includegraphics[width=1\textwidth]{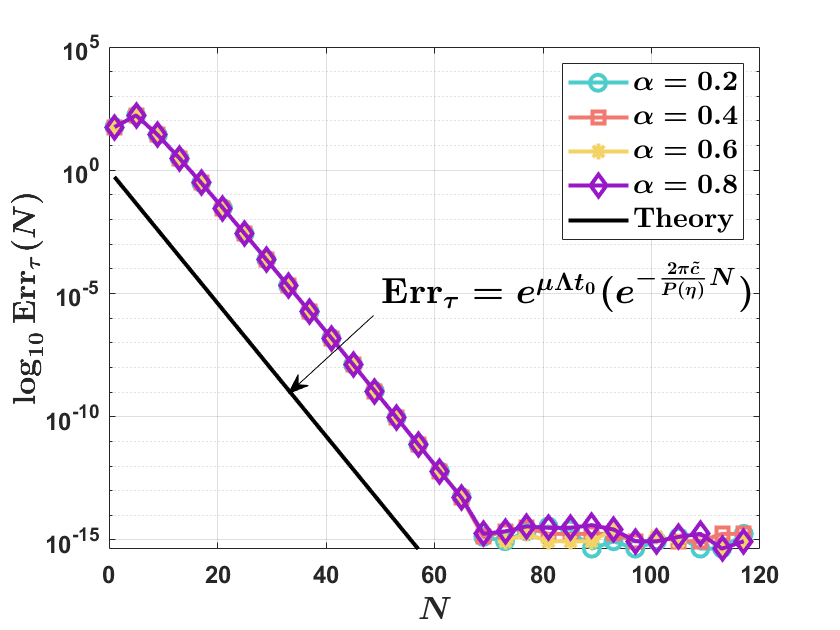}
			\caption{The absolute errors for different $\alpha$ at $t=0.5$ with $\beta=0.77$, $\Lambda=10$.}
			\label{fig:sub3}
		\end{subfigure}
		\hfill
		\begin{subfigure}[c]{0.45\textwidth}
			\includegraphics[width=1\textwidth]{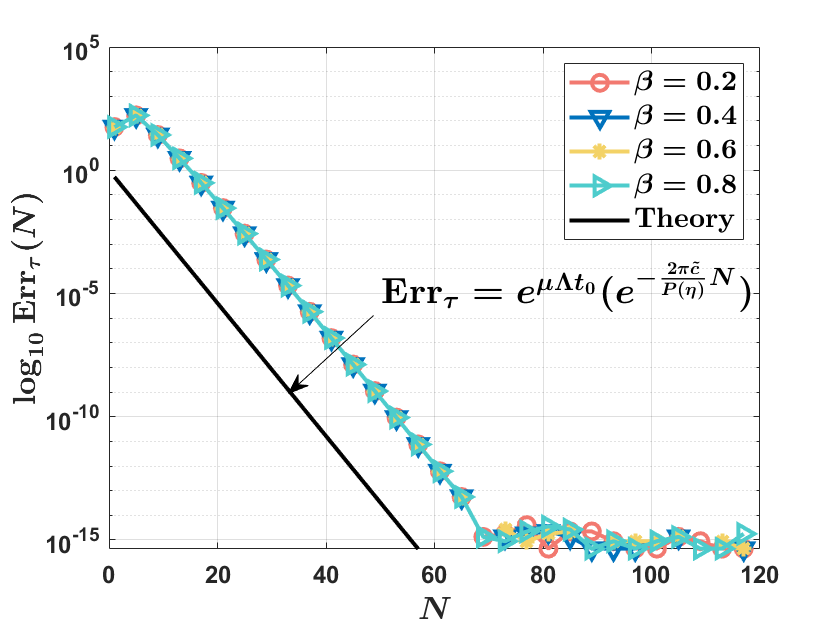}
			\caption{The absolute errors for different $\beta$ at $t=0.5$ with $\alpha=0.2$, $\Lambda=10$.}
			\label{fig:sub4}
		\end{subfigure}
		\caption{Numerical performances of the CIM for  problem (\ref{ex1}).}
		\label{fig.2}
	\end{figure}
	
	The results in Figure \ref{fig.2} demonstrate the excellent performance of the CIM in solving problem (\ref{ex1}), confirming its spectral accuracy. Here, \(N\) denotes the number of CIM quadrature nodes. Figure \ref{fig.2}(a) shows that for a fixed time \(t\), a larger value of \(\Lambda\) leads to a slower convergence rate of the absolute error. Conversely, from Figure \ref{fig.2}(b), one can observe that for a fixed \(\Lambda\), the absolute error remains uniform across different times \(t\). These observations regarding the error behavior and convergence rate are consistent with the theoretical analysis presented in Theorem \ref{Terr2}. Furthermore, Figures \ref{fig.2}(c) and \ref{fig.2}(d) indicate that the CIM exhibits stability and robustness for various values of \(\alpha\) and \(\beta\).
	
	\subsection{Homogeneous Problems with Nonsmooth Initial Values in 1-D and 2-D}
	
	\begin{example}
		Consider the homogeneous problem with nonsmooth initial values for problem (\ref{eq1.1}). In the 1-D case, the initial value is given by
		\begin{equation} \label{eq4.2}
			u_0(x) = \pi^3 \chi_{(0,2/3]}(x) \in L^2(\Omega),
		\end{equation}
		with source term $f(x,t)\equiv 0$. In the 2-D case, the initial value is taken as
		\begin{equation} \label{eq4.3}
			u_0(x,y)=\chi_{(1/2,1)\times(0,1)}(x,y)\in L^2(\Omega),
		\end{equation}
		and the source term is $f(x,y,t)\equiv 0$.
	\end{example}
	
	To evaluate the numerical performance, we use the numerical solution computed with \( N = 200 \) and \( h = 1/2^7 \) as the reference solution. The numerical results at \( t = 0.4 \) are presented in Tables \ref{tab.1}-\ref{tab.4} and Figures \ref{fig.ex2.1}. Additionally, Figures \ref{fig.ex2.2} and \ref{fig.ex2.3} illustrate the behaviors of reference solution at different times, which also confirm that the evolutionary behavior of the solution is influenced by the regularity of the initial condition, as discussed in Subsection \ref{subsec:3.1}. It also indirectly indicates that our algorithm is capable of accurately capturing this instantaneous evolutionary process near $t=0$. For a better observation of the solution's variation, the parameters are taken as $t_0 = 0.001$ and $\Lambda=1000$ in these two figures.
	
	\begin{table}[!ht]
		\centering
		\caption{Temporal errors $\mathrm{Err}_\tau(N)$ at $t = 0.4$ for the 1-D homogeneous problem with nonsmooth initial data \eqref{eq4.2}, shown for different values of $\alpha$ and $\beta$ with $h = 1/2^7$ and $\Lambda = 10$.}
		\renewcommand\arraystretch{1.5}
		\tabcolsep= 0.4 cm
		\begin{tabular}{c|cccccc}
			\Xhline{1.5pt}
			$N$ & ($\alpha,\beta$) & $\mathrm{Err}_\tau(N)$ & ($\alpha,\beta$) & $\mathrm{Err}_\tau(N)$ & ($\alpha,\beta$) & $\mathrm{Err}_\tau(N)$ \\ \hline
			20 & ~ & 6.8676e-04 & ~ & 6.4066e-04 & ~ & 4.8969e-04 \\
			40 & ~ & 6.4736e-08 & ~ & 5.2304e-08 & ~ & 2.8234e-08 \\
			60 & (0.4,0.25) & 5.0307e-12 & (0.5,0.5) & 3.6183e-12 & (0.6,0.75) & 1.5081e-12 \\
			80 & ~ & 1.1514e-13 & ~ & 6.1303e-14 & ~ & 1.0155e-13 \\
			100 & ~ & 1.7835e-13 & ~ & 1.7090e-13 & ~ & 1.4564e-13 \\ 		
			\Xhline{1.5pt}
		\end{tabular}
		\label{tab.1}
	\end{table}	
	
	\begin{table}[!ht]
		\centering
		\caption{Spatial errors $\mathrm{Err}_h(t)$ at $t = 0.4$ for the 1-D homogeneous problem with nonsmooth initial data \eqref{eq4.2}, computed for different values of $\alpha$ and $\beta$ with $N = 200$ and $\Lambda = 10$.}
		\renewcommand\arraystretch{1.5}
		\tabcolsep= 0.4 cm
		\begin{tabular}{c|cccccc}
			\Xhline{1.5pt}
			~ & ~ & ~ & ($\alpha,\beta$) & ~ & ~ & ~ \\
			\Xcline{2-7}{0.5pt}
			~ & (0.4,0.25) & ~ & (0.5,0.5) & ~ & (0.6,0.75) & ~ \\
			$h$ & $ \mathrm{Err}_h(t)$ & Order & $ \mathrm{Err}_h(t)$ & Order & $ \mathrm{Err}_h(t)$ & Order \\
			\Xhline{0.5pt}
			$1/2^5$ & 2.0007e-06 & -- & 4.2965e-07 & -- & 2.8437e-06 & -- \\
			$1/2^6$ & 5.0018e-07 & 2.0000 & 1.0742e-07 & 1.9999 & 7.1095e-07 & 1.9999 \\
			$1/2^7$ & 1.2513e-07 & 1.9990 & 2.6743e-08 & 2.0060 & 1.7765e-07 & 2.0007 \\
			$1/2^8$ & 3.1022e-08 & 2.0121 & 6.5801e-09 & 2.0230 & 4.4663e-08 & 1.9919 \\
			\Xhline{1.5pt}
		\end{tabular}
		\label{tab.2}
	\end{table}
	
	\begin{figure}[htbp]
		\centering
		\begin{minipage}[t]{0.45\textwidth}
			\centering
			\includegraphics[width=1\textwidth]{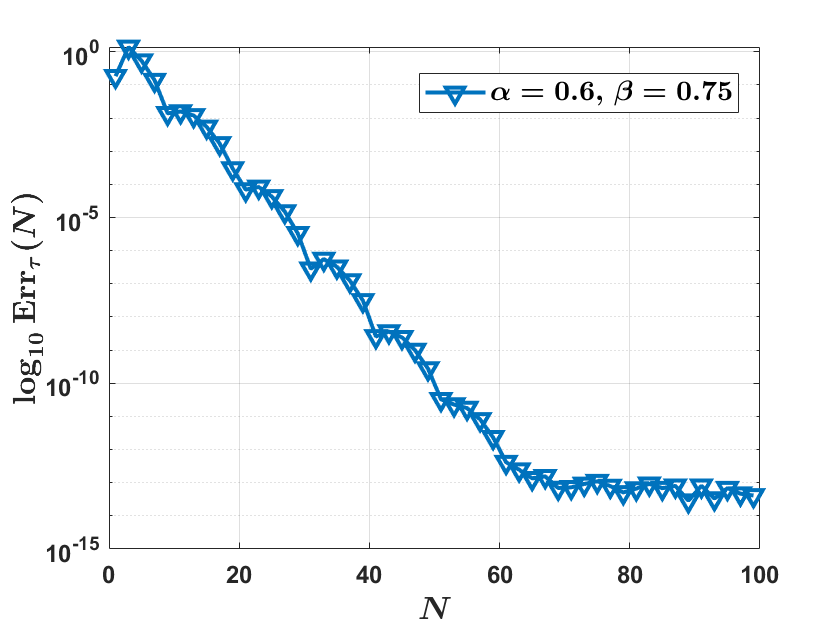}
			\caption{Numerical performance for the 1-D homogeneous problem with nonsmooth initial (\ref{eq4.2}) at $t=0.4$ with $h=1/2^7$, $\alpha=0.6$, $\beta=0.75$, and $\Lambda=10$.}
			\label{fig.ex2.1}
		\end{minipage}
		\hfill
		\begin{minipage}[t]{0.45\textwidth}
			\centering
			\includegraphics[width=1\textwidth]{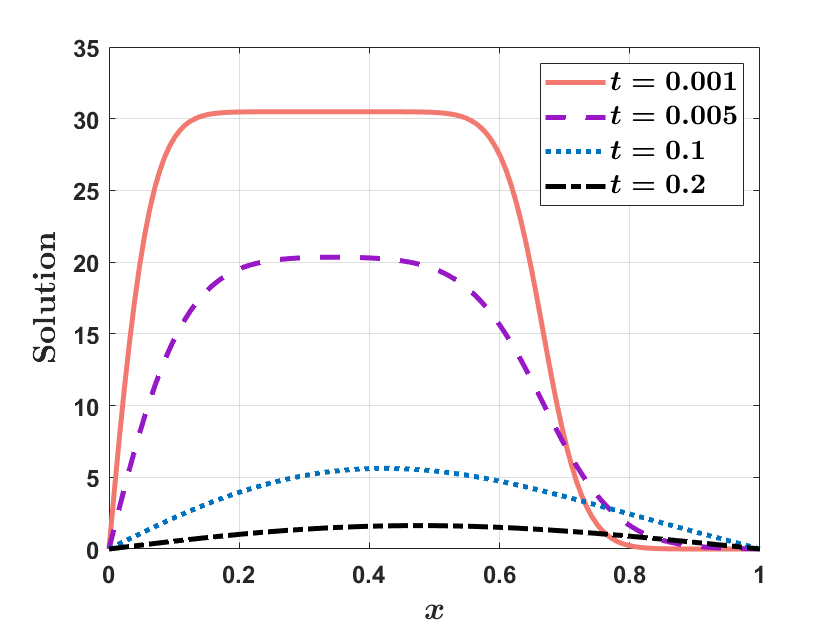}
			\caption{Reference solutions for the 1-D homogeneous problem with nonsmooth initial data (\ref{eq4.2}) at different times, with $h=1/2^7$, $\alpha=0.6$, $\beta=0.75$, and $\Lambda=1000$, $t_0=0.001$.}	
			\label{fig.ex2.2}	
		\end{minipage}
	\end{figure}
	
	\begin{table}[!ht]
		\centering
		\caption{Temporal errors $\mathrm{Err}_\tau(N)$ at $t = 0.4$ for the 2-D homogeneous problem with nonsmooth initial data \eqref{eq4.3}, shown for different values of $\alpha$ and $\beta$ with $h = 1/2^7$ and $\Lambda = 10$.}
		\renewcommand\arraystretch{1.5}
		\tabcolsep= 0.4 cm
		\begin{tabular}{c|cccccc}
			\Xhline{1.5pt}
			$N$ & ($\alpha,\beta$) & $\mathrm{Err}_\tau(N)$ & ($\alpha,\beta$) & $\mathrm{Err}_\tau(N)$ & ($\alpha,\beta$) & $\mathrm{Err}_\tau(N)$ \\ \hline
			40 & ~ & 3.5970e-02 & ~ & 3.5120e-02 & ~ & 2.9018e-02 \\
			60 & ~ & 2.2002e-04 & ~ & 1.9877e-04 & ~ & 2.1639e-05 \\
			80 & (0.4,0.25) & 1.2120e-06 & (0.5,0.5) & 1.3255e-06 & (0.6,0.75) & 2.1043e-06 \\
			100 & ~ & 3.4292e-09 & ~ & 3.2213e-09 & ~ & 1.2120e-09 \\
			120 & ~ & 5.5641e-13 & ~ & 3.4956e-13 & ~ & 7.3516e-13 \\
			\Xhline{1.5pt}
		\end{tabular}
		\label{tab.3}
	\end{table}
	
	\begin{table}[!ht]
		\centering
		\caption{Temporal errors $\mathrm{Err}_\tau(N)$ at $t = 0.1$ for the 2-D homogeneous problem with nonsmooth initial data \eqref{eq4.3}, shown for different values of $\alpha$ and $\beta$ with $h = 1/2^7$ and $\Lambda = 10$.}
		\renewcommand\arraystretch{1.5}
		\tabcolsep= 0.4 cm
		\begin{tabular}{c|cccccc}
			\Xhline{1.5pt}
			$N$ & ($\alpha,\beta$) & $\mathrm{Err}_\tau(N)$ & ($\alpha,\beta$) & $\mathrm{Err}_\tau(N)$ & ($\alpha,\beta$) & $\mathrm{Err}_\tau(N)$ \\ \hline
			40 & ~ & 3.4990E-02 & ~ & 3.4217-02 & ~ & 3.0117-02 \\
			60 & ~ & 2.1973e-04 & ~ & 1.9764e-04 & ~ & 2.0901e-05 \\
			80 & (0.4,0.25) & 1.1975e-06 & (0.5,0.5) & 1.3231e-06 & (0.6,0.75) & 2.0996e-06 \\
			100 & ~ & 3.3116e-09 & ~ & 3.0674e-09 & ~ & 9.9801e-10 \\
			120 & ~ & 5.3112e-13 & ~ & 3.1768e-13 & ~ & 6.9989e-13 \\
			\Xhline{1.5pt}
		\end{tabular}
		\label{tab.t=0.1}
	\end{table}
	
	\begin{table}[!ht]
		\centering
		\caption{Spatial errors $\mathrm{Err}_h(t)$ at $t = 0.4$ for the 2-D homogeneous problem with nonsmooth initial data \eqref{eq4.3}, computed for different values of $\alpha$ and $\beta$ with $N = 200$ and $\Lambda = 10$.}
		\renewcommand\arraystretch{1.5}
		\tabcolsep= 0.4 cm
		\begin{tabular}{c|cccccc}
			\Xhline{1.5pt}
			~ & ~ & ~ & ($\alpha,\beta$) & ~ & ~ & ~ \\
			\Xcline{2-7}{0.5pt}
			~ & (0.4,0.25) & ~ & (0.5,0.5) & ~ & (0.6,0.75) & ~ \\
			$h$ & $ \mathrm{Err}_h(t)$ & Order & $ \mathrm{Err}_h(t)$ & Order & $ \mathrm{Err}_h(t)$ & Order \\
			\Xhline{0.5pt}
			$1/2^5$ & 6.5608e-05 & -- & 5.4822e-05 & -- & 6.1102e-05 & -- \\
			$1/2^6$ & 1.6402e-05 & 2.0000 & 1.4641e-05 & 1.9047 & 1.5345e-05 & 1.9934 \\
			$1/2^7$ & 4.0344e-06 & 2.0235 & 3.5779e-06 & 2.0328 & 3.8398e-06 & 1.9987 \\
			$1/2^8$ & 9.9364e-07 & 2.0215 & 8.8333e-07 & 2.0181 & 9.5990e-07 & 2.0001 \\
			\Xhline{1.5pt}
		\end{tabular}
		\label{tab.4}
	\end{table}
	
	\begin{figure}[htbp]
		\centering
		\begin{subfigure}[c]{0.45\textwidth}
			\includegraphics[width=1\textwidth]{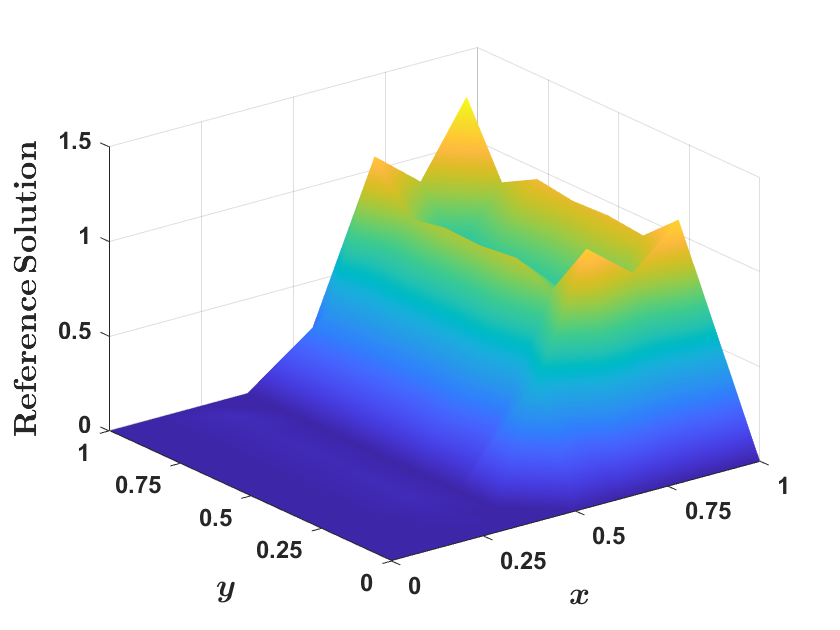}
			\caption{$t=0.001$.}
		\end{subfigure}
		\hfill
		\begin{subfigure}[c]{0.45\textwidth}
			\includegraphics[width=1\textwidth]{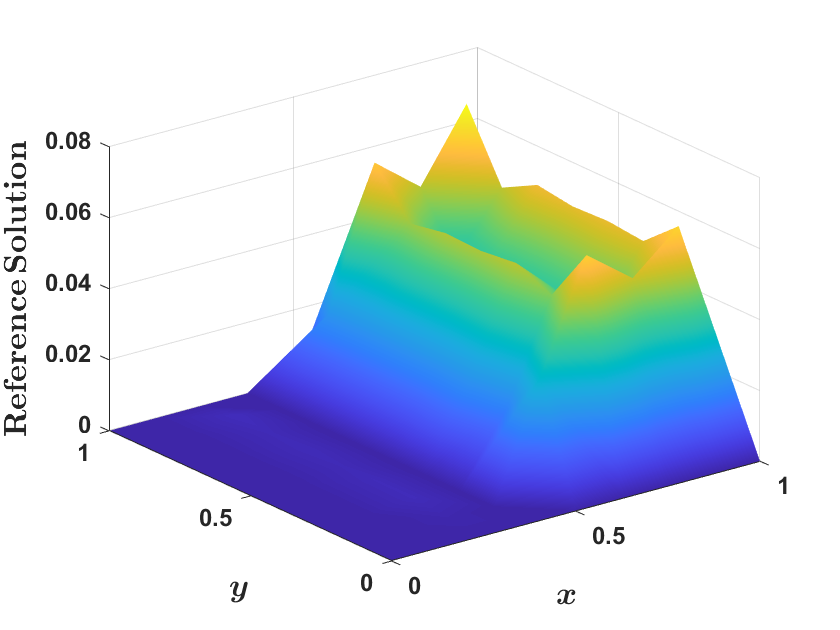}
			\caption{$t=0.02$.}
		\end{subfigure}
		\hfill
		\begin{subfigure}[c]{0.45\textwidth}
			\includegraphics[width=1\textwidth]{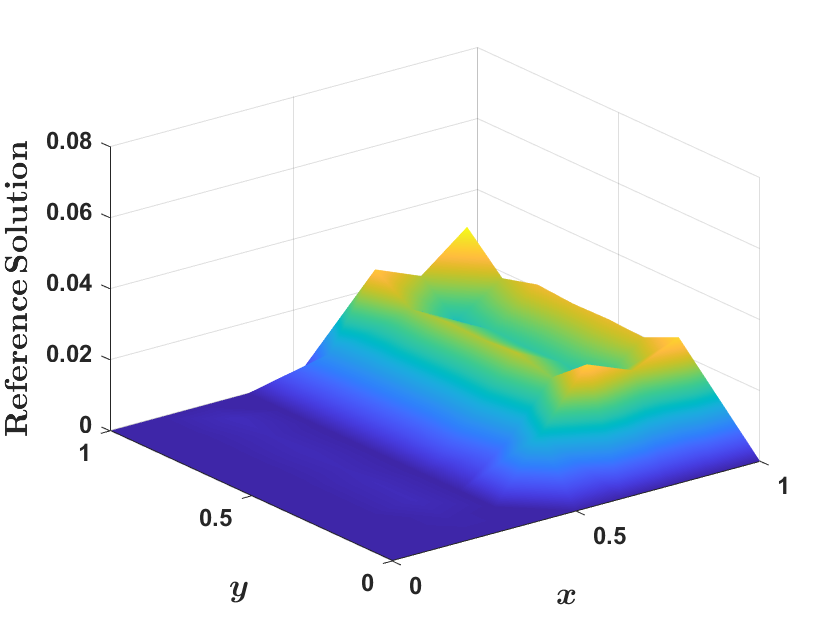}
			\caption{$t=0.04$.}
		\end{subfigure}
		\hfill
		\begin{subfigure}[c]{0.45\textwidth}
			\includegraphics[width=1\textwidth]{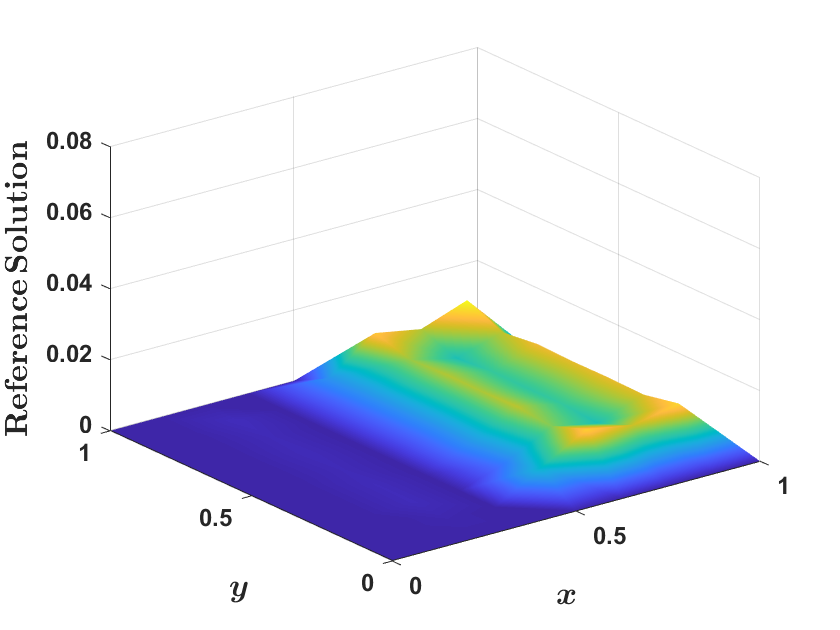}
			\caption{$t=0.1$.}
		\end{subfigure}
		\caption{Reference solutions for the 2-D homogeneous problem with  nonsmooth initial value (\ref{eq4.3}) at different times, with $\alpha=0.6$, $\beta=0.75$, and $\Lambda=1000$, $t_0=0.001$.}
		\label{fig.ex2.3}
	\end{figure}
	
	Tables \ref{tab.1} and \ref{tab.3}, along with Figure \ref{fig.ex2.1}, demonstrate that the CIM achieves spectral convergence and high accuracy in the temporal direction. Furthermore, Tables \ref{tab.2} and \ref{tab.4} confirm that the spatial convergence order of the proposed algorithm is \(O(h^2)\) for various values of \(\alpha\) and \(\beta\). These results are in full agreement with the theoretical analysis established in Theorem \ref{thm3}.
	
	Moreover, compared with other algorithms, such as those in \cite{mahata2023nonsmooth}, our method attains superior precision. In particular, for the 2-D case, our algorithm delivers enhanced accuracy at a comparable spatial computational cost. Additionally, as evidenced by the comparison between Tables \ref{tab.3} and \ref{tab.t=0.1}, the computational expense of our algorithm remains independent of the final time \(t\), since \(t\) is treated as a parameter. \textbf{\textit{This underscores a key advantage of our approach: it consistently maintains high accuracy even for problems with nonsmooth initial data.}}
	
	\subsection{Nonhomogeneous Problems with Vanishing Initial Values in 1-D and 2-D}
	
	\begin{example}
		In the 1-D case, we take the initial value $u_0(x)\equiv 0$, with the source term
		\begin{equation} \label{eq4.4}
			f(x,t)=\frac{1}{\Gamma(2-\beta)}x^3(1-x)t^{1-\beta}+(12x^2-6x)\left(t + \frac{1}{\Gamma(\alpha+2)}t^{\alpha+1}\right),
		\end{equation}
		thus the exact solution is
		$$
		u(x,t)=tx^3(1-x).
		$$
		In the 2-D case, we take $u_0(x,y)=0$ with the constant source term
		\begin{equation} \label{eq4.5}
			f(x,y,t)=1.
		\end{equation}
	\end{example}
	
	The reference solution is computed numerically using $N=200$, $h=1/2^9$. The corresponding results are presented in Tables \ref{tab.6}-\ref{tab.9} and Figure \ref{fig.7}.
	
	\begin{table}[!ht]
		\centering
		\caption{Temporal errors $\mathrm{Err}_\tau(N)$ at $t = 0.6$ for the 1-D nonhomogeneous problem with source term \eqref{eq4.4}, shown for different values of $\alpha$ and $\beta$ with $h = 1/2^{10}$ and $\Lambda = 10$.}
		\renewcommand\arraystretch{1.5}
		\tabcolsep= 0.4 cm
		\begin{tabular}{c|cccccc}
			\Xhline{1.5pt}
			$N$ & ($\alpha,\beta$) & $\mathrm{Err}_\tau(N)$ & ($\alpha,\beta$) & $\mathrm{Err}_\tau(N)$ & ($\alpha,\beta$) & $\mathrm{Err}_\tau(N)$ \\ \hline
			10 & ~ & 4.5381e-03 & ~ & 4.1416e-03 & ~ & 3.8911e-03 \\
			20 & ~ & 5.7582e-05 & ~ & 4.7460e-05 & ~ & 5.6301e-05 \\
			30 & (0.4,0.25) & 5.6110e-07 & (0.5,0.5) & 1.4720e-07 & (0.6,0.75) & 2.2743e-07 \\
			40 & ~ & 1.3154e-08 & ~ & 2.4810e-08 & ~ & 1.8756e-08 \\
			60 & ~ & 1.0070e-08 & ~ & 3.0381e-08 & ~ & 2.4069e-08 \\ 	
			\Xhline{1.5pt}
		\end{tabular}
		\label{tab.6}
	\end{table}
	
	\begin{table}[!ht]
		\centering
		\caption{Spatial errors $\mathrm{Err}_h(t)$ at $t = 0.6$ for the 1-D nonhomogeneous problem with source term \eqref{eq4.4}, computed for different values of $\alpha$ and $\beta$ with $N = 80$ and $\Lambda = 10$.}
		\renewcommand\arraystretch{1.5}
		\tabcolsep= 0.4 cm
		\begin{tabular}{c|cccccc}
			\Xhline{1.5pt}
			~ & ~ & ~ & ($\alpha,\beta$) & ~ & ~ & ~ \\
			\Xcline{2-7}{0.5pt}
			~ & (0.4,0.25) & ~ & (0.5,0.5) & ~ & (0.6,0.75) & ~ \\
			$h$ & $ \mathrm{Err}_h(t)$ & Order & $ \mathrm{Err}_h(t)$ & Order & $ \mathrm{Err}_h(t)$ & Order \\
			\Xhline{0.5pt}
			$1/2^{5}$ & 6.4321e-06 & -- & 7.9102e-06 & -- & 6.1210e-06 & -- \\
			$1/2^{6}$ & 1.6542e-06 & 1.9610 & 2.0583e-06 & 1.9457 & 1.5938e-06 & 1.9472 \\
			$1/2^{7}$ & 1.7421e-07 & 2.0015 & 2.2184e-07 & 2.0001 & 1.7240e-07 & 2.0003 \\
			$1/2^{8}$ & 1.9823e-08 & 2.0450 & 2.5089e-08 & 2.0673 & 1.9543e-08 & 2.0581 \\
			\Xhline{1.5pt}
		\end{tabular}
		\label{tab.7}
	\end{table}
	
	\begin{figure}[htbp]
		\centering
		\begin{subfigure}[c]{0.45\textwidth}
			\includegraphics[width=1\textwidth]{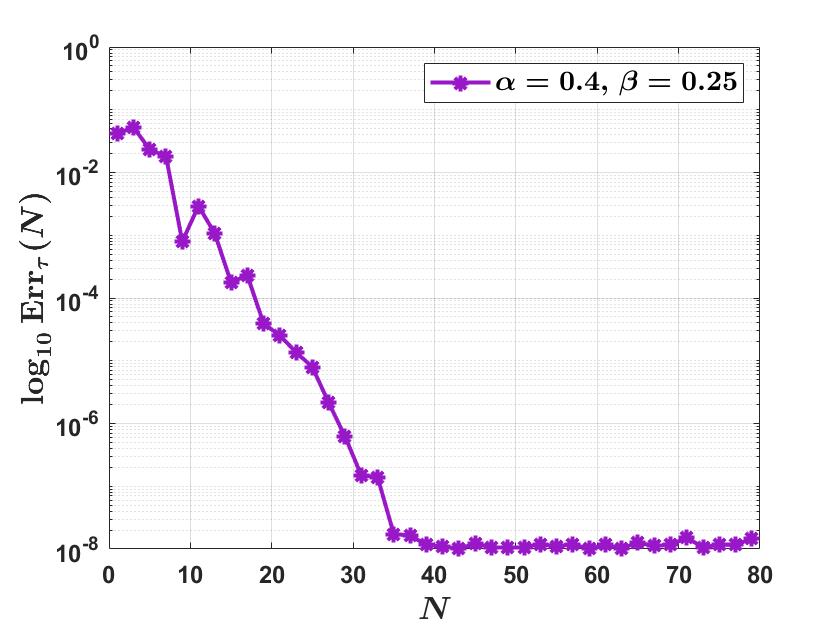}
			\caption{Absolute errors at $t = 0.6$ with $\alpha = 0.4$, $\beta = 0.25$, and $\Lambda = 10$.}
		\end{subfigure}
		\hfill
		\begin{subfigure}[c]{0.45\textwidth}
			\includegraphics[width=1\textwidth]{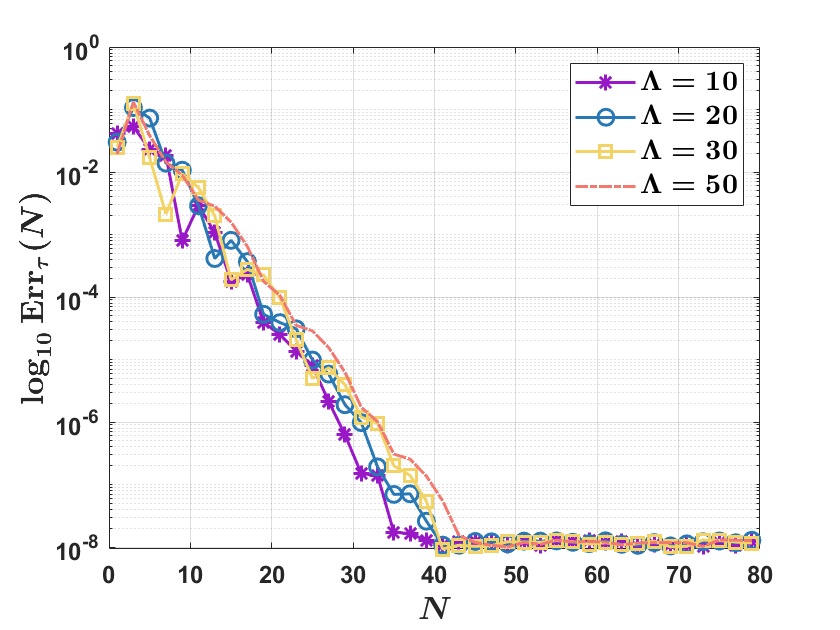}
			\caption{Absolute errors at $t = 0.6$ for different values of $\Lambda$, with $\alpha = 0.4$ and $\beta = 0.25$.}
		\end{subfigure}
		\caption{Numerical performances for the 1-D nonhomogeneous problem with source term \eqref{eq4.4}.}
		\label{fig.7}
	\end{figure}	
	
	\begin{table}[!ht]
		\centering
		\caption{Temporal errors $\mathrm{Err}_\tau(N)$ at $t = 0.6$ for the 2-D nonhomogeneous problem with source term \eqref{eq4.5}, shown for different values of $\alpha$ and $\beta$ with $h = 1/2^9$ and $\Lambda = 10$.}
		\renewcommand\arraystretch{1.5}
		\tabcolsep= 0.4 cm
		\begin{tabular}{c|cccccc}
			\Xhline{1.5pt}
			$N$ & ($\alpha,\beta$) & $\mathrm{Err}_\tau(N)$ & ($\alpha,\beta$) & $\mathrm{Err}_\tau(N)$ & ($\alpha,\beta$) & $\mathrm{Err}_\tau(N)$ \\ \hline
			40 & ~ & 1.4395e-02 & ~ & 1.5526e-02 & ~ & 1.7467e-02 \\
			60 & ~ & 3.6818e-04 & ~ & 4.0467e-04 & ~ & 4.6063e-05 \\
			80 & (0.4,0.25) & 6.5819e-06 & (0.5,0.5) & 5.5642e-06 & (0.6,0.75) & 4.0130e-06 \\
			100 & ~ & 2.5370e-08 & ~ & 2.1251e-08 & ~ & 1.4846e-08 \\
			120 & ~ & 8.9230e-11 & ~ & 7.4987e-11 & ~ & 5.2070e-11 \\ 		
			\Xhline{1.5pt}
		\end{tabular}
		\label{tab.8}
	\end{table}
	
	\begin{table}[!ht]
		\centering
		\caption{Spatial errors $\mathrm{Err}_h(t)$ at $t = 0.6$ for the 2-D nonhomogeneous problem with source term \eqref{eq4.5}, computed for different values of $\alpha$ and $\beta$ with $N = 200$ and $\Lambda = 10$.}
		\renewcommand\arraystretch{1.5}
		\tabcolsep= 0.4 cm
		\begin{tabular}{c|cccccc}
			\Xhline{1.5pt}
			~ & ~ & ~ & ($\alpha,\beta$) & ~ & ~ & ~ \\
			\Xcline{2-7}{0.5pt}
			~ & (0.4,0.25) & ~ & (0.5,0.5) & ~ & (0.6,0.75) & ~ \\
			$h$ & $ \mathrm{Err}_h(t)$ & Order & $ \mathrm{Err}_h(t)$ & Order & $ \mathrm{Err}_h(t)$ & Order \\
			\Xhline{0.5pt}
			$1/2^5$ & 1.8123e-06 & -- & 2.2370e-06 & -- & 3.4003e-06 & -- \\
			$1/2^6$ & 4.5307e-07 & 2.0000 & 5.5926e-07 & 1.9999 & 8.5018e-07 & 1.9998 \\
			$1/2^7$ & 1.1327e-07 & 2.0000 & 1.3982e-07 & 2.0000 & 2.1255e-07 & 1.9999 \\
			$1/2^8$ & 2.8310e-08 & 2.0003 & 3.4939e-08 & 2.0007 & 5.3104e-08 & 2.0009 \\
			\Xhline{1.5pt}
		\end{tabular}
		\label{tab.9}
	\end{table}
	
	Tables \ref{tab.6} and \ref{tab.8}, along with Figure \ref{fig.7}, demonstrate the spectral accuracy of the proposed scheme in the temporal direction. Furthermore, Tables \ref{tab.7} and \ref{tab.9} show that the spatial error converges at second order for various values of $\alpha$ and $\beta$. These results are consistent with the theoretical analysis presented in Theorem \ref{thm4}.
	
	\subsection{Nonhomogeneous Problem with Nonsmooth Solution in the 1-D Case}
	
	\begin{example}
		We consider a 1-D nonhomogeneous case of problem (\ref{eq1.1}) with a nonsmooth solution. The exact solution is chosen as
		\begin{equation}
			\label{exact}
			u(x,t) = 1 + t^{1/6}x^3(1-x),
		\end{equation}
		thus the source term is
		$$
		f(x,t) = 1+\frac{\Gamma(\frac{7}{6})}{\Gamma(\frac{7}{6}-\beta)}x^3(1-x)t^{\frac{1}{6}-\beta} + (12x^2-6x) \left(t^{\frac{1}{6}} + \frac{\Gamma(\frac{7}{6})}{\Gamma(\alpha+\frac{7}{6})} t^{\alpha+\frac{1}{6}}\right),
	    $$
		and the Laplace transform of $f(x,t)$ with respect to $t$ is
		$$
		\widehat{f}(x,z) = z^{-1}+ \Gamma\Big(\frac{7}{6}\Big) \left( x^3(1-x)z^{\beta-\frac{7}{6}} + (12x^2-6x)\left(z^{-\frac{7}{6}} + z^{-\alpha-\frac{7}{6}}\right) \right).
		$$
		Numerical results are presented in Table \ref{tab.10} and Figure \ref{fig.8}.
	\end{example}
	
	\begin{table}[!ht]
		\centering
		\caption{Temporal errors $\mathrm{Err}_\tau(N)$ at $t = 0.5$ for the 1-D nonhomogeneous problem with nonsmooth solution \eqref{exact}, shown for different values of $\alpha$ and $\beta$ with $h = 1/2^{10}$ and $\Lambda = 5$.}
		\label{tab.10}
		\renewcommand\arraystretch{1.5}
		\tabcolsep= 0.4 cm
		\begin{tabular}{c|cccccc}
			\Xhline{1.5pt}
			$N$ & ($\alpha,\beta$) & $\mathrm{Err}_\tau(N)$ & ($\alpha,\beta$) & $\mathrm{Err}_\tau(N)$ & ($\alpha,\beta$) & $\mathrm{Err}_\tau(N)$ \\ \hline
			10 & ~ & 1.6714e-03 & ~ & 1.9751e-03 & ~ & 1.7548e-03 \\
			20 & ~ & 2.8019e-05 & ~ & 3.0776e-05 & ~ & 2.0855e-05 \\
			30 & (0.25,0.4) & 1.4872e-07 & (0.5,0.6) & 1.3044e-07 & (0.75,0.8) & 1.7584e-07 \\
			40 & ~ & 1.2034e-08 & ~ & 2.0471e-08 & ~ & 1.2758e-08 \\
			60 & ~ & 1.9078e-08 & ~ & 1.3649e-08 & ~ & 1.1572e-08 \\
			\Xhline{1.5pt}
		\end{tabular}
	\end{table}
	
	\begin{figure}[htbp]
		\centering
		\begin{subfigure}[c]{0.45\textwidth}
			\includegraphics[width=1\textwidth]{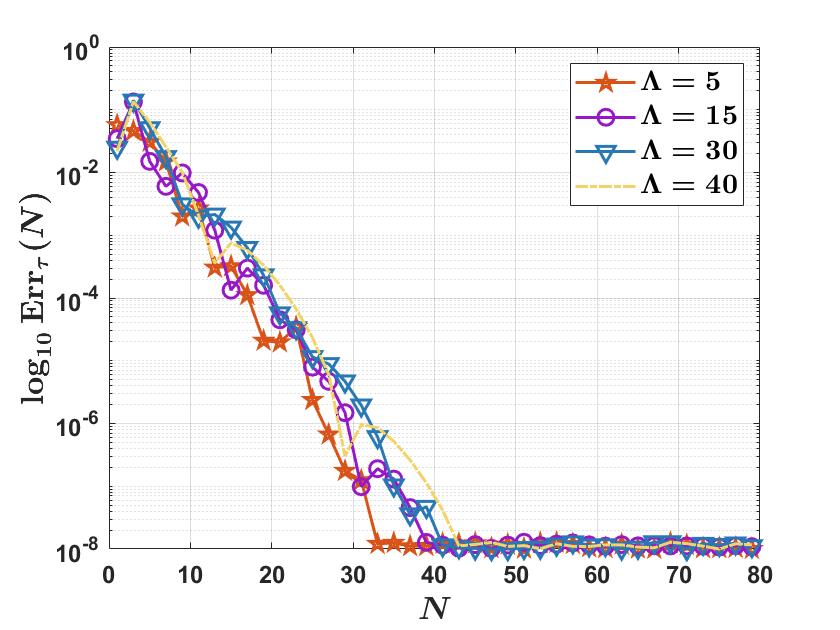}
			\caption{Absolute errors at $t = 0.5$ for different values of $\Lambda$, with $\alpha = 0.75$ and $\beta = 0.8$.}
		\end{subfigure}
		\hfill
		\begin{subfigure}[c]{0.45\textwidth}
			\includegraphics[width=1\textwidth]{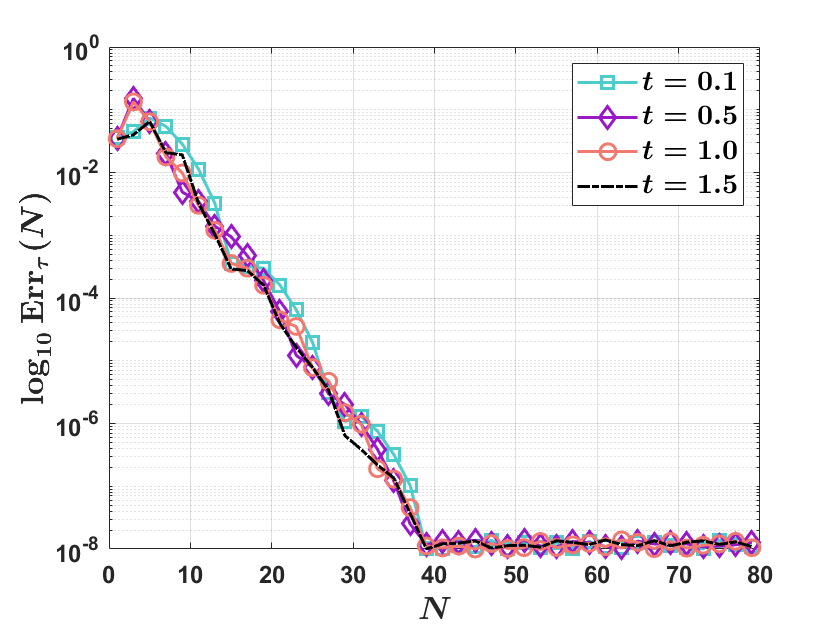}
			\caption{Absolute errors at different times for $\Lambda = 15$, with $\alpha = 0.25$ and $\beta = 0.4$.}
		\end{subfigure}
		\caption{Numerical performances for the 1-D nonhomogeneous problem with nonsmooth solution \eqref{exact}.}
		\label{fig.8}
	\end{figure}
	
	As shown in Table \ref{tab.10}, \textbf{\textit{the proposed algorithm maintains spectral convergence even for problems with nonsmooth solutions, demonstrating another superiority of our method.}} This observation is further supported by Figure \ref{fig.8}. Similar to the results for the scalar problem, the CIM remains stable across different values of \(\alpha\) and \(\beta\). While, different from the scalar case, as shown in these plots, the temporal error does not continue to decrease with increasing quadrature nodes
	$N$, due to the limiting effect of spatial error.
	
	\section{Conclusions}
	\label{section5}
	
	In this paper, we propose an efficient localization method with low regularity requirements for solving time-fractional integro-differential equations. A rigorous error analysis is provided, demonstrating that the scheme achieves spectral accuracy in the temporal direction. The localization of the nonlocal operator in time substantially improves computational efficiency and facilitates parallel computation in spatial approximations. The effectiveness and robustness of the proposed scheme are thoroughly validated through numerous numerical examples. These experiments, conducted on both one-dimensional and two-dimensional domains under various conditions, including nonsmooth initial data, vanishing initial values, as well as nonsmooth solutions, consistently confirm the scheme’s spectral accuracy, low regularity requirements, and minimal memory consumption. Future work will extend this approach to nonlinear problems, particularly semi-linear equations, building upon the foundation established here.
	


\begin{thebibliography}{10}
		\bibitem{Lazarov2015An}
		Bazhlekova, E., Jin, B., Lazarov, R., Zhou, Z.
		An analysis of the Rayleigh-Stokes problem for a generalized second-grade fluid.
		\textit{Numer. Math.}, 131: 1-31, 2015.
		\bibitem{berrut2004barycentric}
		Berrut, J.-P., Trefethen, L. N.
		Barycentric Lagrange interpolation.
		\textit{SIAM Rev.}, 46(3): 501-517, 2004.
		\bibitem{brenner2008mathematical}
		Brenner, S. C., Scott, L. R.
		\textit{The Mathematical Theory of Finite Element Methods}.
		3rd ed. Springer, New York, 2008.
		\bibitem{caputo2000models}
		Caputo, M.
		Models of flux in porous media with memory.
		\textit{Water Resour. Res.}, 36(3): 693-705, 2000.
		\bibitem{cuesta2006convolution}
		Cuesta, E., Lubich, C., Palencia, C.
		Convolution quadrature time discretization of fractional diffusion-wave equations.
		\textit{Math. Comp.}, 75(254): 673-696, 2006.
		\bibitem{deng2019high}
		Deng, W., Zhang, Z.
		\textit{High Accuracy Algorithm for the Differential Equations Governing Anomalous Diffusion: Algorithm and Models for Anomalous Diffusion}.
		World Scientific, Singapore, 2019.
		\bibitem{2008Some}
		Gorenflo, R., Mainardi, F.
		Some recent advances in theory and simulation of fractional diffusion processes.
		\textit{J. Comput. Appl. Math.}, 229(2): 400-415, 2009.
		\bibitem{in2011contour}
		In't Hout, K. J., Weideman, J. A. C.
		A contour integral method for the Black-Scholes and Heston equations.
		\textit{SIAM J. Sci. Comput.}, 33(2): 763-785, 2011.
		\bibitem{kato1961fractional}
		Kato, T.
		Fractional powers of dissipative operators.
		\textit{J. Math. Soc. Japan}, 13(3): 246-274, 1961.
		\bibitem{krasnoschok2017semilinear}
		Krasnoschok, M., Pata, V., Vasylyeva, N.
		Semilinear subdiffusion with memory in the one-dimensional case.
		\textit{Nonlinear Anal.}, 165: 1-17, 2017.
		\bibitem{krasnoschok2019semilinear}
		Krasnoschok, M., Pata, V., Vasylyeva, N.
		Semilinear subdiffusion with memory in multidimensional domains.
		\textit{Math. Nachr.}, 292(7): 1490-1513, 2019.
		\bibitem{lee2006parallel}
		Lee, J., Sheen, D.
		A parallel method for backward parabolic problems based on the Laplace transformation.
		\textit{SIAM J. Numer. Anal.}, 44(4): 1466-1486, 2006.
		\bibitem{lee2013laplace}
		Lee, H., Lee, J., Sheen, D.
		Laplace transform method for parabolic problems with time-dependent coefficients.
		\textit{SIAM J. Numer. Anal.}, 51(1): 112-125, 2013.
		\bibitem{lin2007finite}
		Lin, Y., Xu, C.
		Finite difference/spectral approximations for the time-fractional diffusion equation.
		\textit{J. Comput. Phys.}, 225(2): 1533-1552, 2007.
		\bibitem{li2021second}
		Li, X., Liao, H., Zhang, L.
		A second-order fast compact scheme with unequal time-steps for subdiffusion problems.
		\textit{Numer. Algorithms}, 86: 1011-1039, 2021.
		\bibitem{li2022exponential}
		Li, B., Ma, S.
		Exponential convolution quadrature for nonlinear subdiffusion equations with nonsmooth initial data.
		\textit{SIAM J. Numer. Anal.}, 60(2): 503-528, 2022.
		\bibitem{LopezFernndez2004OnTN}
		López-Fernández, M., Palencia, C.
		On the numerical inversion of the Laplace transform of certain holomorphic mappings.
		\textit{Appl. Numer. Math.}, 51(2-3): 289-303, 2004.
		\bibitem{lopez2006spectral}
		López-Fernández, M., Palencia, C., Schädle, A.
		A spectral order method for inverting sectorial Laplace transforms.
		\textit{SIAM J. Numer. Anal.}, 44(3): 1332-1350, 2006.
		\bibitem{luo2022numerical}
		Luo, Z., Zhang, X., Wang, S., Yao, L.
		Numerical approximation of time fractional partial integro-differential equation based on compact finite difference scheme.
		\textit{Chaos Solitons Fractals}, 161: 112395, 2022.
		\bibitem{ma2023analyses}
		Ma, F., Zhao, L., Deng, W., Wang, Y.
		Analyses of the contour integral method for time fractional normal-subdiffusion transport equation.
		\textit{J. Sci. Comput.}, 97(2): 45, 2023.
		\bibitem{Mahata2021FiniteEM}
		Mahata, S., Sinha, R. K.
		Finite element method for fractional parabolic integro-differential equations with smooth and nonsmooth initial data.
		\textit{J. Sci. Comput.}, 87(1): 7, 2021.
		\bibitem{mahata2023nonsmooth}
		Mahata, S., Sinha, R. K.
		Nonsmooth data error estimates of the L1 scheme for subdiffusion equations with positive-type memory term.
		\textit{IMA J. Numer. Anal.}, 43(3): 1742-1778, 2023.
		
		\bibitem{martensen1968numerischen}
		Martensen, E.
		Zur numerischen auswertung uneigentlicher integrale.
		\textit{Z. Angew. Math. Mech}, 48:T83-T85, 1968.
		
		\bibitem{stenger2012numerical}
		Stenger, F.
		\textit{Numerical methods based on Sinc and analytic functions},
		Volume 20.
		Springer Science \& Business Media, 2012.
		
		\bibitem{mclean2010maximum}
		McLean, W., Thomée, V.
		Maximum-norm error analysis of a numerical solution via Laplace transformation and quadrature of a fractional-order evolution equation.
		\textit{IMA J. Numer. Anal.}, 30(1): 208-230, 2010.
		\bibitem{McLean2006ASA}
		McLean, W., Mustapha, K.
		A second-order accurate numerical method for a fractional wave equation.
		\textit{Numer. Math.}, 105(3): 481-510, 2006.
		\bibitem{mohebbi2017compact}
		Mohebbi, A.
		Compact finite difference scheme for the solution of a time fractional partial integro-differential equation with a weakly singular kernel.
		\textit{Math. Methods Appl. Sci.}, 40(18): 7627-7639, 2017.
		\bibitem{mosqueira2020antibody}
		Mosqueira, A., Camino, P. A., Barrantes, F. J.
		Antibody-induced crosslinking and cholesterol-sensitive, anomalous diffusion of nicotinic acetylcholine receptors.
		\textit{J. Neurochem.}, 152(6): 663-674, 2020.
		\bibitem{mustapha2020l1}
		Mustapha, K.
		An L1 approximation for a fractional reaction-diffusion equation, a second-order error analysis over time-graded meshes.
		\textit{SIAM J. Numer. Anal.}, 58(2): 1319-1338, 2020.
		\bibitem{2009Discontinuous}
		Mustapha, K., McLean, W.
		Discontinuous Galerkin method for an evolution equation with a memory term of positive type.
		\textit{Math. Comp.}, 78(268): 1975-1995, 2009.
		
		\bibitem{Podlubny:1999}
		Podlubny, I. Fractional Differential Equations. Academic Press, San Diego, 1999.
		\bibitem{Riewe1997Mechanics}
		Riewe, F.
		Mechanics with fractional derivatives.
		\textit{Phys. Rev. E}, 55(3): 3581-3592, 1997.
		\bibitem{sheen2003parallel}
		Sheen, D., Sloan, I. H., Thomée, V.
		A parallel method for time discretization of parabolic equations based on Laplace transformation and quadrature.
		\textit{IMA J. Numer. Anal.}, 23(2): 269-299, 2003.
		\bibitem{stynes2017error}
		Stynes, M., O'Riordan, E., Gracia, J. L.
		Error analysis of a finite difference method on graded meshes for a time-fractional diffusion equation.
		\textit{SIAM J. Numer. Anal.}, 55(2): 1057-1079, 2017.
		\bibitem{2006A}
		Sun, Z. Z., Wu, X.
		A fully discrete difference scheme for a diffusion-wave system.
		\textit{Appl. Numer. Math.}, 56(2): 193-209, 2006.
		\bibitem{talbot1979accurate}
		Talbot, A.
		The accurate numerical inversion of Laplace transforms.
		\textit{IMA J. Appl. Math.}, 23(1): 97-120, 1979.
		\bibitem{wang2021sharp}
		Wang, Z., Cen, D., Mo, Y.
		Sharp error estimate of a compact L1-ADI scheme for the two-dimensional time-fractional integro-differential equation with singular kernels.
		\textit{Appl. Numer. Math.}, 159: 190-203, 2021.
		\bibitem{weideman2006optimizing}
		Weideman, J. A. C.
		Optimizing Talbot’s contours for the inversion of the Laplace transform.
		\textit{SIAM J. Numer. Anal.}, 44(6): 2342-2362, 2006.
		\bibitem{weideman2007parabolic}
		Weideman, J. A. C., Trefethen, L. N.
		Parabolic and hyperbolic contours for computing the Bromwich integral.
		\textit{Math. Comp.}, 76(259): 1341-1356, 2007.
		\bibitem{weideman2019gauss}
		Weideman, J. A. C.
		Gauss-Hermite quadrature for the Bromwich integral.
		\textit{SIAM J. Numer. Anal.}, 57(5): 2200-2216, 2019.
		\bibitem{zaeri2017fractional}
		Zaeri, S., Saeedi, H., Izadi, M.
		Fractional integration operator for numerical solution of the integro-partial time fractional diffusion heat equation with weakly singular kernel.
		\textit{Asian-Eur. J. Math.}, 10(4): 1750071, 2017.
		\bibitem{zhou2020alternating}
		Zhou, J., Xu, D.
		Alternating direction implicit difference scheme for the multi-term time-fractional integro-differential equation with a weakly singular kernel.
		\textit{Comput. Math. Appl.}, 79(2): 244-255, 2020.
		
	\end{thebibliography}

\end{document}